\documentclass[12pt, twoside]{article}
\usepackage{amsmath,amsthm,amssymb}
\usepackage{times}
\usepackage{enumerate}

\RequirePackage[numbers]{natbib}
\RequirePackage[colorlinks,citecolor=red,urlcolor=red]{hyperref}

\def\titlerunning#1{\gdef\titrun{#1}}
\makeatletter
\def\author#1{\gdef\autrun{\def\and{\unskip, }#1}\gdef\@author{#1}}
\def\address#1{{\def\and{\\\hspace*{18pt}}\renewcommand{\thefootnote}{}%
\footnote {#1}}%
\markboth{\autrun}{\titrun}}
\makeatother
\def\email#1{e-mail: #1}
\def\subjclass#1{{\renewcommand{\thefootnote}{}%
\footnote{\emph{Mathematics Subject Classification (2010):} #1}}}
\def\keywords#1{\par\medskip
\noindent\textbf{Keywords.} #1}

\newtheorem{thm}{Theorem}[section]
\newtheorem{cor}[thm]{Corollary}
\newtheorem{lem}[thm]{Lemma}

\newtheorem{prp}[thm]{Proposition}

\theoremstyle{definition}

\newtheorem{rem}[thm]{Remark}

\numberwithin{equation}{section}

\usepackage{xcolor}

\begin{document}

\newcommand{\Q}{\mathbb{Q}}
\newcommand{\R}{\mathbb{R}}
\newcommand{\Rr}{\overline{\mathbb{R}}}
\newcommand{\I}{\mathbb{I}}
\newcommand{\Z}{\mathbb{Z}}
\newcommand{\N}{\mathbb{N}}
\newcommand{\F}{\mathcal{F}}
\newcommand{\p}{\mathbb{P}}
\newcommand{\B}{\mathcal{B}}
\newcommand{\M}{\mathcal{M}}
\newcommand{\A}{\mathcal{A}}
\newcommand{\Af}{\mathfrak{A}}
\newcommand{\Pp}{\mathcal{P}}
\newcommand{\pp}{\mathfrak{P}}
\newcommand{\s}{\mathcal{S}}
\newcommand{\G}{\mathcal{G}}
\newcommand{\RR}{\mathcal{R}}
\newcommand{\E}{\mathbb{E}}
\newcommand{\h}{\mathcal{H}}
\newcommand{\En}{\mathrm{E}^n}
\newcommand{\eps}{\varepsilon}
\newcommand{\supp}{\mathrm{supp\,}}
\newcommand{\law}{\mathrm{Law}}
\newcommand{\sgn}{\mathrm{sgn\,}}
\newcommand{\pr}{\mathrm{pr}}
\newcommand{\Ss}{\mathfrak{S}}
\newcommand{\tr}{\bigtriangleup}
\newcommand{\St}{\mathrm{St}}
\newcommand{\leb}{\mathrm{Leb}}
\newcommand{\Int}{\mathrm{Int\,}}
\newcommand{\Var}{\mathrm{Var\,}}


\baselineskip=17pt


\titlerunning{On asymptotic behavior of the modified Arratia flow}

\title{On asymptotic behavior of \\ the modified Arratia flow}

\author{Vitalii Konarovskyi}


\maketitle

\address{Universit\"{a}t Leipzig, Fakult\"{a}t f\"{u}r Mathematik und Informatik, Augustusplatz 10, 04109 Leipzig, Germany; \email{konarovskiy@gmail.com}}



\begin{abstract}
We study asymptotic properties of the system of interacting diffusion particles on the real line which transfer a mass~\cite{Konarovskyi:2014:arx}. The system is a natural generalization of the coalescing Brownian motions~\cite{Arratia:1979,Le_Jan:2004}. The main difference is that diffusion particles coalesce summing their mass and changing their diffusion rate inversely proportional to the mass. First we construct  the system in the case where the initial mass distribution has the moment of the order greater then two as an $L_2$-valued martingale with a suitable quadratic variation. Then we find the relationship between the asymptotic behavior of the particles and local properties of the mass distribution at the initial time.

\end{abstract}

\section{Introduction}

In the paper we study local properties of the modified Arratia flow. The flow is a variant of the Arratia
flow~\cite{Arratia:1979,Dorogovtsev:2004,Le_Jan:2004} for a system of Brownian motions on the real line which move independently up to their meeting and then coalesce. The fundamental new feature is that particles
carry mass which is aggregated as particles coalesce and which determines the diffusivity of the individual particle in an inverse proportional way. The modified Arratia flow was first constructed in~\cite{Konarovskiy:2010:TVP:en} (see also~\cite{Konarovskiy:2010:UMJ:en,Konarovskiy:2011:TSP,Konarovskyi:2013:COSA,Konarovskyi:2014:TSP}), as a physical generalization of the system of coalescing Brownian motions, in the case where particles start from integer points with unit masses. Later in~\cite{Konarovskyi:2014:arx} the modified Arratia flow for a system of particles which start from all points of the interval $[0,1]$ with zero mass (the distribution of the mass of particles at the initial time is the Lebesgue measure on $[0,1]$) was constructed as a scaling limit.

The first main result of the paper is the generalization of the model constructed in~\cite{Konarovskyi:2014:arx} to the case of any mass distribution of particles at the start. Using martingale methods, we prove the following theorem.

\begin{thm}\label{theorem_modification}
 For each $b>0$ and non-decreasing c\`{a}dl\'{a}g function $g$ satisfying
 \begin{equation}\label{f_L_2_eps_cond}
\int_0^b|g(u)|^{2+\eps}du<\infty
\end{equation}
for some $\eps>0$, there exists a process $\{X(u,t),\ u\in(0,b),\ t\in[0,T]\}$ from Skorohod space $D((0,b),C[0,T])$ such that
 \begin{enumerate}
\item[$(C1)$] for all $u\in(0,b)$, $X(u,\cdot)$ is a continuous square integrable martingale with respect to the filtration
$$
\F_t=\sigma(X(u,s),\ s\leq t,\ u\in(0,b)),\quad t\in[0,T];
$$

\item[$(C2)$] $X(u,0)=g(u)$ for all $u\in(0,b)$;

\item[$(C3)$] $X(u,t)\leq X(v,t)$ for all $u<v$ and $t\in[0,T]$;

\item[$(C4)$]  for all $t\in[0,T]$ and $u,v\in(0,b)$ the joint quadratic variation
$$
\langle X(u,\cdot),X(v,\cdot)\rangle_t=\int_0^t\frac{\I_{\{\tau_{u,v}\leq s\}}ds}{m(u,s)},
$$
where $m(u,t)=\leb\{w:\ \exists s\leq t\ X(u,s)=X(w,s)\}$ and
$\tau_{u,v}=\inf\{t:\ X(u,t)=X(v,t)\}\wedge T$.
\end{enumerate}
\end{thm}

The process $X$ describes the evolution of particles with the mass distribution $\mu$ at the start, where $\mu$ is the push forward of the Lebesgue measure on $[0,b]$, i.e
\begin{equation}\label{f_push_forward}
\mu=g_\#\leb\big|_{[0,b]}.
\end{equation}
The following lemma explains that.

\begin{lem}
 Let $A=\{g(u),\ u\in(0,b)\}$ and $X$ from $D((0,b),C[0,T])$ satisfy $(C1)-(C4)$. Then the family of processes
 \begin{equation}\label{f_process_z}
  Z(g(u),\cdot)=X(u,\cdot),\quad u\in(0,b),
 \end{equation}
 is well-defined and satisfies \begin{enumerate}
\item[$(A1)$] for all $x\in A$ the process $Z(x,\cdot)$ is a continuous
square integrable martingale with respect to the filtration
$$
\sigma(Z(x,s),\ x\in A,\ s\leq t),\quad t\in[0,T];
$$
\item[$(A2)$] for all $x\in A$, $Z(x,0)=x$;

\item[$(A3)$] for all $x<y$ from $A$ and $t\in[0,T]$, $Z(x,t)\leq Z(y,t)$;

\item[$(A4)$] for all $x,y\in A$ the joint quadratic variation
$$
\left\langle Z(x,\cdot),Z(y,\cdot)\right\rangle_t=\int_0^t\frac{\I_{\{\tau_{x,y}^{\mu}\leq
s\}}ds}{m_{\mu}(x,s)},
$$
where $m_{\mu}(x,t)=\mu\{z:\ \exists s\leq t\ \ Z(z,s)=Z(x,s)\}$ and $\tau_{x,y}^{\mu}=\inf\{t:\
Z(x,t)=Z(y,t)\}\wedge T$.
\end{enumerate}
\end{lem}

\begin{proof}
Since for $g(u)=g(v)$ we have $X(u,\cdot)=X(v,\cdot)$ (it follows from Remark~\ref{remark_conditions_C_implies_mart_prop_in_L} and propositions~\ref{prp_values_in_St} and~\ref{prp_modif_from_D_for_mart} below), the process $Z$ is well-defined. Moreover, if $x=g(u)$, then we have
$$
Z(x,0)=Z(g(u),0)=X(u,0)=g(u)=x
$$
and by~\eqref{f_push_forward},
\begin{align*}
 m_{\mu}(x,t)&=\mu\{z:\ \exists s\leq t\ \ Z(z,s)=Z(x,s)\}\\
 &=\leb\{v:\ \exists s\leq t\ \ Z(g(v),s)=X(u,s)\}\\
 &=\leb\{v:\ \exists s\leq t\ \ X(v,s)=X(u,s)\}=m(u,t).
\end{align*}
Thus, $Z$, defined by~\eqref{f_process_z}, satisfies $(A1)-(A4)$.
\end{proof}

So, we see that interpreting $Z(u,t)$ as the position of the particle at time $t$ starting from $u$, the family of processes $\{Z(u,\cdot), u\in A\}$ is a description of the system of particles which start from almost all points of $\supp\mu$ with the mass distribution $\mu$. Although it seems that $Z$ gives a simpler description of the model, it is easier to work with the process $X$. Firstly, the values of the random variable $X(\cdot,t)$ are functions defined on the interval $(0,b)$, where the interval is independent of the support of the initial distribution $\mu$ (it only depends on the total mass of the system). Consequently, the particle system can be approximated by finite subsystems on the same state space. Secondly, $X$ is an $L_2^{\uparrow}$-valued continuous martingale with the quadratic variation $\langle X\rangle_t=\int_0^t\pr_{X(s)}ds$, where $L_2^{\uparrow}$ is the set of all non-decreasing functions from $L_2$ and $\pr_fh$ denotes the projection of $h$ in $L_2$ on the subspace of $\sigma(f)$-measurable functions. Moreover, we will show that each $L_2^{\uparrow}$-valued continuous martingale $\widehat{X}$ with the quadratic variation $\langle \widehat{X}\rangle_t=\int_0^t\pr_{\widehat{X}(s)}ds$ has a modification that satisfies the same properties as $X$ (see~Theorem~\ref{thm_conditions_C}). Thus, to construct the modified Arratia flow it is enough to construct an $L_2^{\uparrow}$-valued continuous martingale with the needed quadratic variation.

The second main result of the paper is a relationship between local properties of the distribution of particle mass at the start and asymptotic behavior of individual particles and its masses for small time. Using estimations of the expectations of particle mass and particle diffusion rate (see Section~\ref{mass_estimations}) and also the law of the iterated logarithm for the Wiener process, we prove the following statements.

\begin{thm}\label{prp_LIL_1}
 Let $\alpha>\frac{1}{2}$, $u_0\in(0,1)$ and there exist $C>0$ and $\delta>0$ such that the following assumptions hold
 \begin{enumerate}
  \item[(i)] $|g(u-\varsigma)-g(u)|\leq C |u-u_0|^{(\alpha-1)\vee 0}|\varsigma|^{\alpha\wedge 1}$ for all $u\in[u_0-\delta,u_0+\delta]$ and all $\varsigma$ between $0$ and $u-u_0$;

  \item[(ii)] $|g(u)-g(u_0)|\geq \frac{1}{C}|u-u_0|^{\alpha}$ for all $u\in[u_0-\delta,u_0+\delta]$.
 \end{enumerate}
 Then for all $\epsilon>0$
 \begin{align}
  \p\left\{\lim_{t\to 0}\frac{m(u_0,t)}{t^{\frac{1}{2\alpha+1}}\left(\ln\frac{1}{t}\right)^{1+\epsilon}}=0\right\}=1, \label{f_LIL_m_1}\\
  \p\left\{\varlimsup_{t\to 0}\frac{|X(u_0,t)-g(u_0)|}{t^{\frac{\alpha}{2\alpha+1}}\left(\ln\frac{1}{t}\right)^{-\frac{1}{2}-\epsilon}}=+\infty\right\}=1.\label{f_LIL_Y_1}
 \end{align}
\end{thm}

\begin{thm}\label{prp_LIL_2}
 Let $u_0\in(0,1)$, $\alpha>\frac{1}{2}$ and there exist $\delta>0$ and $C>0$ such that $g(u_0+u)-g(u_0)\leq Cu^{\alpha}$, $u\in[0,\delta]$, or $g(u_0)-g(u_0-u)\leq Cu^{\alpha}$, $u\in[0,\delta]$. Then for all $\epsilon>0$
 \begin{align*}
  \p\left\{\lim_{t\to 0}\frac{m(u_0,t)}{t^{\frac{1}{2\alpha+1}}\left(\ln\frac{1}{t}\right)^{-1-\epsilon}}=+\infty\right\}=1,\\
  \p\left\{\lim_{t\to 0}\frac{|X(u_0,t)-g(u_0)|}{t^{\frac{\alpha}{2\alpha+1}}\left(\ln\frac{1}{t}\right)^{\frac{1}{2}+\epsilon}}=0\right\}=1.
 \end{align*}
\end{thm}

\begin{rem}
 If $\alpha\geq 1$ and $g$ is differentiable in a neighborhood of $u_0$  with
 $$
 \lim_{u\to u_0}\frac{g'(u)}{|u-u_0|^{\alpha-1}}=C\in(0,\infty),
 $$
 then $g$ satisfies assumptions $(i)$, $(ii)$ of Theorem~\ref{prp_LIL_1}.
\end{rem}

\begin{rem}
 If $\alpha>\frac{1}{2}$ and
 $$
 g(u)=\sgn(u-u_0)|u-u_0|^{\alpha},\quad u\in(0,1),
 $$
 then $g$ also satisfies assumptions $(i)$, $(ii)$ of Theorem~\ref{prp_LIL_1}.
\end{rem}

In particular, theorems~\ref{prp_LIL_1},~\ref{prp_LIL_2} imply that the modified Arratia flow constructed in~\cite{Konarovskyi:2014:arx} (where $g(u)=u$, $u\in[0,1]$) has the following behavior
 \begin{align*}
  \p\left\{\lim_{t\to 0}\frac{m(u,t)}{\sqrt[3]{t}\left(\ln\frac{1}{t}\right)^{1+\epsilon}}=0\right\}
  &=\p\left\{\lim_{t\to 0}\frac{m(u,t)}{\sqrt[3]{t}\left(\ln\frac{1}{t}\right)^{-1-\epsilon}}=+\infty\right\}=1,\\
  \p\left\{\lim_{t\to 0}\frac{|X(u,t)-u|}{\sqrt[3]{t}\left(\ln\frac{1}{t}\right)^{\frac{1}{2}+\epsilon}}=0\right\}
  &=\p\left\{\varlimsup_{t\to 0}\frac{|X(u,t)-u|}{\sqrt[3]{t}\left(\ln\frac{1}{t}\right)^{-\frac{1}{2}-\epsilon}}=+\infty\right\}=1
 \end{align*}
for all $u\in(0,1)$ (see Remark~\ref{rem_asymprotic_bahavior_of_Z}). 

We note that the asymptotic behavior of each particle in the Arratia flow~$\{a(u,t),\ u\in\R,\ t\geq 0\}$ is as follows
$$
\p\left\{\varlimsup_{t\to 0}\frac{|a(u,t)-u|}{\sqrt{2t\ln\ln\frac{1}{t}}}=1\right\}=1,
$$
since each process $a(u,\cdot)$ is a Brownian motion with unit diffusion rate. Moreover, the process $\nu(t)=\leb\{u:\ \exists s\leq t\ \ a(u,s)=a(0,s)\},\ t\geq 0$, that describes the cluster size (it corresponds to the particle mass in our case), has the following behavior~\cite{Vovchanskii:2012}
\begin{align*}
\p\left\{\varlimsup\limits_{t\to 0}\frac{\nu(t)}{\sqrt{2t\,\ln\ln\frac{1}{t}}}\geq 1\right\}=1,\\
\p\left\{\varlimsup\limits_{t\to 0}\frac{\nu(t)}{2\sqrt{t\,\ln\ln\frac{1}{t}}}\leq 1\right\}=1.
\end{align*}
Comparing the behavior of particles and their masses in the modified Arratia flow with the behavior of particles in the Arratia flow, we see that asymptotics are completely different, since the diffusion rates of particles in the first case grow to infinity and make particles to fluctuate more and more intensively for small time.

Here we would like to note that many methods which work for studying of local properties of the Arratia flow do not work in our case, since they are based on the fact that every system of particles can be considered separately from the whole system. Therefore, the Arratia flow can be investigated just by studying of its finite subsystems (see, e.g~\cite{Ostapenko:2010,Shamov:2011,Chernega:2012,Dorogovtsev:2014}). There is an opposite situation for studying of the modified Arratia flow, where every finite subsequence cannot be considered as a separate system.

The modified Arratia flow has a connection with the Wasserstein diffusion, constructed by M.-K.~von~Renessse and T.~Sturm in~\cite{Renesse:2009} (see also~\cite{Andres:2010,Stannat:2013,Sturm:2014}). In fact, in~\cite{Konarovskyi_LDP:2015} V.~Konarovskyi and M.-K.~von~Renesse proved that the process describing the evolution of particle mass in the modified Arratia flow solves a SPDE that is similar to the SPDE for the Wasserstein diffusion and also showed via a large deviation analysis that the flow satisfies the Varadhan formula with the square of the Wasserstein distance as the rate function.
Namely, if $\{X(u,t),\ u\in[0,1],\ t\in[0,T]\}$ satisfies $(C1)-(C4)$ with $g(u)=u$, $u\in[0,1]$, then the process
$$
\mu_t=X(\cdot,t)_{\#}\leb\big|_{[0,1]},\quad t\in[0,T],
$$
is a weak solution to the equation
$$
d\mu_t =  \Gamma(\mu_t)dt + \mathop{{\rm div}}( \sqrt {\mu_t} dW_t),
$$
where $\Gamma(\nu)$ is defined on test functions as follows $(f,\Gamma(\nu))=\sum_{x \in \mathop{supp}(\nu) } f''(x).$
Moreover, for suitable sets $A \subset \Pp(\R)$ we have
$$
 \lim_{\eps \to 0} \eps\log\p\{\mu_\eps\in A\}=
- \frac {d^2_{\mathcal W}\left(\leb\big|_{[0,1]},A\right)}{2},
$$
where $d_{\mathcal W}$ denotes the (quadratic) Wasserstein metric on the space of probability measures on $\R$. Basically we believe that the same form of the short time behavior of the particle system with the initial particle distribution $\mu$ is valid for any probability measure $\mu$ (with $\int_{\R}|x|^{2+\eps}\mu(dx)<\infty$) instead of $\leb\big|_{[0,1]}$. Thus, the process constructed in the present paper can be considered as a candidate for an intrinsic Brownian motion on the Wasserstein space of probability measures. Consequently, the question of existence such a process is important and its local properties is of interest.

\subsection{Organization of the article.}

In section~\ref{the_main_definitions} we introduce the main notation and formulate some statements about $L_2^{\uparrow}$-valued continuous martingales. In section~\ref{finite_system} a finite system of particles is defined as a continuous martingale taking values in $L_2^{\uparrow}$. The main estimations for the particle system is obtained in section~\ref{estimations}. Section~\ref{construction} is devoted to the construction of an $L_2^{\uparrow}$-valued continuous martingale $X$ which starts from a function $g\in L_{2+\eps}$ and has the quadratic variation $\langle X\rangle_t=\int_0^t\pr_{X(s)}ds$. In section~\ref{modification_from_D} we prove that the martingale $X$ has a modification from the Skorohod space $D((0,b),C[0,T])$ that satisfies similar properties as the flow constructed in~\cite{Konarovskyi:2014:arx}. Section~\ref{mass_estimations} is the key section of the paper. There we obtain estimations of the expectations of mass and diffusion rate of individual particles which allow to state the asymptotic behavior of the particle system in section~\ref{asymptotic_behavior}.

\section{The main definitions}\label{the_main_definitions}

\subsection{Some notation}
For $p\geq 1$ we denote the space of $p$-integrable functions (more precisely equivalence classes) from $[a,b]$ to $\R$ by $L_p[a,b]$ or $L_p$ and $\|\cdot\|_{L_p}$ is the usual norm on $L_p$. Also $(\cdot,\cdot)$ denotes the inner product in $L_2[a,b]$. Let $D^{\uparrow}[a,b]$ or $D^{\uparrow}$ be the set of c\`{a}dl\'{a}g non-decreasing functions from $[a,b]$ into $\Rr=\R\cup\{-\infty,+\infty\}$. For convenience we assume that all functions from $D^{\uparrow}$ are continuous at $b$. Let $L_2^{\uparrow}[a,b]$ or shortly $L_2^{\uparrow}$ be the subset of $L_2[a,b]$ that contains functions (their equivalence classes) from $D^{\uparrow}$, i.e $f\in L_2[a,b]$ belongs to $L_2^{\uparrow}$ if there exists $g\in D^{\uparrow}$ such that $f=g$ a.e. Set $L_p^{\uparrow}=L_2^{\uparrow}\cap L_p$, $p\geq 2$.

Note that $L_p^{\uparrow}$ is a closed subset of $L_p$ (see Corollary~\ref{cor_clasability}). Consequently, $L_p^{\uparrow}$ is a Polish space with respect to the distance induced by $\|\cdot\|_{L_p}$.

Since each function $f$ from $L_p^{\uparrow}$ has a unique modification from $D^{\uparrow}$ (see Remark~\ref{rem_unique_modification}), considering $f$ as a map from $[a,b]$ to $\Rr$, we always take its modification from $D^{\uparrow}$.

For each $f\in L_p^{\uparrow}$, let $\Pi_f$ denote the class of sets $\{v\in[a,b]:\ f(v)=c\}$, $c\in\R$, of the positive length (as we agreed, $f\in D^{\uparrow}$). Since $f$ is a non-decreasing function, elements of $\Pi_f$ are intervals $[c,d)$, $a\leq c<d<b$, and $[c,d]$, $a\leq c<d\leq b$. Moreover, $\Pi_f$ is finite or countable. If $\Pi_f$ is finite and $\bigcup\Pi_f=[a,b]$, then $|\Pi_f|$ denotes the number of elements in $\Pi_f$, otherwise $|\Pi_f|=+\infty$. Let us introduce the partial order for $\Pi_{\cdot}$. We write $\Pi_g\leq\Pi_f$ if for each $\pi\in\Pi_f$ there exists $\pi'\in\Pi_g$ such that $\Int\pi\subseteq\pi'$, where $\Int\pi$ denotes the interior of $\pi$. Let $|\pi|$ denote the length of $\pi$ for $\pi\in\Pi_f$.
\begin{rem}\label{rem_size_of_Pi}
From definitions of $\Pi_{\cdot}$ and $|\Pi_{\cdot}|$ it follows that the inequality $\Pi_g\leq\Pi_f$ implies $|\Pi_g|\leq|\Pi_f|$.
\end{rem}

If $|\Pi_f|<\infty$, then $f$ is called the \textit{step function} ($f$ takes a finite number of values). The set of all step functions (from $D^{\uparrow}$) we denote by $\St$. If $f$ is a step function, then $\Pi_f=\{\pi_1,\ldots,\pi_n\}$, where $\pi_1=[a,a_1),\ \pi_2=[a_1,a_2),\ \ldots\ \pi_{n-1}=[a_{n-2},a_{n-1}),\ \pi_n=[a_{n-1},b]$ for some $a<a_1<\ldots<a_{n-1}<b$. In this case $f=\sum_{k=1}^nx_k\I_{\pi_k}$ for some $x_1<\ldots<x_n$, where $\I_A$ denotes the characteristic function of the set $A$. Henceforth, for $f\in\St$ we numerate elements of $\Pi_f$ in increasing order i.e. writing $\Pi_f=\{\pi_1,\ldots,\pi_n\}$, we mean that elements of $\pi_k$ is less then elements of $\pi_{k+1}$ for all $k\in [n-1]$, where $[n]=\{1,\ldots,n\}$.

\subsection{$L_2^{\uparrow}$-valued martingales}

Let $(\F_t)_{t\in[0,T]}$ be a right continuous filtration on a probability space $(\Omega,\F,\p)$. An $L_2^{\uparrow}$-valued continuous random process $X(t)$, $t\in[0,T]$, given on $(\Omega,\F,\p)$, is called an $(\F_t)$-\textit{square integrable martingale} if it is $(\F_t)$-adapted, $\E\|X(t)\|^2_{L_2}<\infty$, $t\in[0,T]$, and for each $0\leq s<t\leq T$,
$$
\E\left(X(t)|\F_s\right)=X(s).
$$

Let $\E\|X(t)\|^2_{L_2}<\infty$ for all $t\in[0,T]$. Since $L_2^{\uparrow}$ is a subset of the separable Hilbert space $L_2$, $X(t)$, $t\in[0,T]$, is an $(\F_t)$-square integrable martingale if and only if for any $h\in L_2$, $(X(t),h)$, $t\in[0,T]$, is an $(\F_t)$-martingale.

If the filtration $(\F_t)_{t\in[0,T]}$ is generated by $X$, i.e. $\F_t=\bigcap_{\eps>0}\sigma((X(s),h),\ s\leq t+\eps,\ h\in L_2)$, $t\in[0,T)$, and $\F_T=\sigma((X(s),h),\ s\leq T,\ h\in L_2)$, then we will call $X$ just a square integrable martingale.

It is well-known that two real-valued continuous martingales $x_1(t)$, $x_2(t)$, $t\in[0,T]$, satisfying $x_1(t)\leq x_2(t)$ for all $t\in[0,T]$ coincide after their meeting. This property implies that $\Pi_{X(t)}$, $t\geq 0$, decreases a.s.

\begin{prp}\label{prp_coalescing}
Let $X(t)$, $t\in[0,T]$, be an $L_2^{\uparrow}[a,b]$-valued continuous $(\F_t)$-square integrable martingale. Then
$$
\p\left\{\mbox{for all}\  s\leq t,\ \Pi_{X(t)}\leq\Pi_{X(s)} \right\}=1.
$$
\end{prp}

For a Polish space $E$, let $C([0,T],E)$ denote the space of $E$-valued continuous functions on $[0,T]$ with the usual uniform norm $\|\cdot\|_C$. For $E=\R$ we use the notation $C[0,T]$. The set of right continuous $C[0,T]$-valued functions on $(a,b)$ which have left limits is denoted by $D((a,b),C[0,T])$.

\begin{prp}\label{prp_modif_from_D_for_mart}
Let $X(t)$, $t\in[0,T]$, be an $L_2^{\uparrow}[a,b]$-valued continuous $(\F_t)$-square integrable martingale such that for each $t\in(0,T]$, $X(t)\in\St$ a.s.
Then $X$ has a modification from $D((a,b),C[0,T])$, that is, there exists $C[0,T]$-valued random process $\widetilde{X}(u,\cdot)$, $u\in(a,b)$, with trajectories from $D((a,b),C[0,T])$ such that for all $t\in[0,T]$, $X(t)=\widetilde{X}(\cdot,t)$ (in $L_2$) a.s. Moreover, for each $u\in(a,b)$, $\widetilde{X}(u,\cdot)$ is a continuous $(\F_t)$-square integrable martingale and
\begin{equation}\label{f_coalescing}
\p\{\forall u,v\in(a,b)\ \forall s\in[0,T]\ \widetilde{X}(u,s)=\widetilde{X}(v,s)\ \mbox{implies } \widetilde{X}(u,t)=\widetilde{X}(v,t),\ \forall t\geq s\}=1.
\end{equation}
\end{prp}

\begin{proof}
 The propositions are proved in Appendix~\ref{appendix_modif_from_D_for_mart}.
\end{proof}

We define the \textit{quadratic variation} $\langle X\rangle_t$, $t\in[0,T]$, of $X$ as an $(\F_t)$-adapted continuous process starting from zero, with values in the space of nonnegative definite trace-class operators on $L_2$, such that for all $h,g\in L_2$ the joint quadratic variation of the martingales $(X(t),h)$, $(X(t),g)$, $t\in[0,T]$, is
$$
\left\langle(X(\cdot),h),(X(\cdot),g)\right\rangle_t=(\langle X\rangle_t h,g),\quad t\in[0,T].
$$
For more details we refer to~\cite{Gawarecki:2011}.

\section{A finite system of particles}\label{finite_system}

In this section we construct an $L_2^{\uparrow}[a,b]$-valued square integrable martingale with the suitable quadratic variation that describes the evolution of a finite system of coalescing diffusion particles.
Let the system of processes $\{x_k(t),\ t\in[0,T],\ k\in[d]\}$ describe the evolution of particles which start from points $x_1^0<x_2^0<\ldots<x_d^0$ with masses $m_1^0,m_2^0,\ldots,m_d^0$. Such a system of processes has been constructed e.g. in~\cite{Konarovskiy:2010:UMJ:en} and satisfies the following properties
\begin{enumerate}
\item[$(F1)$] for all $k\in[d]$, $x_k(\cdot)$ is a continuous square integrable martingale with respect to the filtration
$$
\F_t^d=\bigcap_{\eps>0}\sigma(x_k(s),\ s\leq t+\eps,\ k\in[d]),\quad t\in[0,T];
$$

\item[$(F2)$] $x_k(0)=x_k^0$ for all $k\in[d]$;

\item[$(F3)$] $x_k(t)\leq x_l(t)$ for all $k<l$ and $t\in[0,T]$;

\item[$(F4)$]  for all $t\in[0,T]$
$$
\langle x_k(\cdot),x_l(\cdot)\rangle_t=\int_0^t\frac{\I_{\{\tau_{k,l}\leq s\}}ds}{m_k(s)},
$$
where $m_k(t)=\sum_{i\in A_k(t)}m_i^0$, $A_k(t)=\{i:\ \exists s\leq t\ x_k(s)=x_i(s)\}$ and
$\tau_{k,l}=\inf\{t:\ x_k(t)=x_l(t)\}\wedge T$.
\end{enumerate}
Moreover, $(F1)-(F4)$ uniquely determine the distribution of the system that is stated in the following lemma.
\begin{lem}\label{lem_uniqueness_x}
If systems of processes $\{x_k(t),\ t\in[0,T],\ k\in[d]\}$ and $\{x'_k(t),\ t\in[0,T],\ k\in[d]\}$ satisfy $(F1)-(F4)$, then their distributions coincide.
\end{lem}
\begin{proof}
 The proof is similar to the proof of Lemma~3~\cite{Konarovskiy:2010:TVP:en}.
\end{proof}

Let us construct an $L_2^{\uparrow}$-valued process that corresponds to the system $\{x_k(t),\ t\in[0,T],\ k\in[d]\}$. Set $a_0=0$, $a_k=a_{k-1}+m_k^0$, $k\in[d]$, and $b=a_d$. Let $\pi_k=[a_{k-1},a_k)$, $k\in[d-1]$, and $\pi_d=[a_{d-1},b]$. We take
\begin{equation}\label{f_function_g}
g=\sum_{k=1}^dx_k^0\I_{\pi_k}
\end{equation}
and
\begin{equation}\label{f_connection_X_x}
X(t)=\sum_{k=1}^dx_k(t)\I_{\pi_k},\quad t\in[0,T].
\end{equation}
It is obvious that $X$ is an $L_2^{\uparrow}[0,b]$-valued continuous process which starts from $g$. Since $\|X(t)\|_{L_2}^2=\sum_{k=1}^dx_k^2(t)|\pi_k|$, we have $\E\|X(t)\|_{L_2}^2<\infty$. Next, for any $h\in L_2$
$$
(X(t),h)=\sum_{k=1}^dx_k(t)(\I_{\pi_k},h),\quad t\in[0,T],
$$
is a martingale. Consequently, $X$ is a square integrable martingale. Let us evaluate its quadratic variation.

Denote the projection of $h$ in $L_2$ on the subspace of $\sigma(g)$-measurable functions by $\pr_gh$. If $g$ is defined by~\eqref{f_function_g}, then
\begin{equation}\label{f_projection}
\pr_gh=\sum_{k=1}^d\frac{1}{|\pi_k|}(\I_{\pi_k},h)\I_{\pi_k}.
\end{equation}
Using properties $(F1)-(F4)$, similarly to~\cite{Konarovskyi_LDP:2015} one can show that
\begin{enumerate}
 \item[(M)] $\langle(X(\cdot),h)\rangle_t=\int_0^t\|\pr_{X(s)}h\|_{L_2}^2ds$ for all $h\in L_2$.
\end{enumerate}
By the polarization formulas for the inner product $(\cdot,\cdot)$ and the joint quadratic variation $\langle\cdot,\cdot\rangle_{\cdot}$, we obtain for $h,f\in L_2$
$$
\left\langle(X(\cdot),h),(X(\cdot),f)\right\rangle_t=\int_0^t(\pr_{X(s)}h,\pr_{X(s)}f)ds=\int_0^t(\pr_{X(s)}h,f)ds,\quad t\in[0,T].
$$
Thus, we have shown that $X$ is an $L_2^{\uparrow}$-valued continuous square integrable martingale with the quadratic variation
\begin{enumerate}
 \item[(M')]$\langle X\rangle_t=\int_0^t\pr_{X(s)}ds$.
\end{enumerate}

We note that $\int_0^t\pr_{X(s)}ds$ is a trace-class operator, since $X$ is a square integrable martingale~~\cite[Lemma~2.1]{Gawarecki:2011}. It follows also from the fact that $\pr_{X(s)}$ is a projection on a space with dimension smaller or equal than $d$ for all $s\in[0,T]$.

Next we prove the inverse statement.
\begin{lem}\label{lem_connection_between_X_x}
 Let $b$, $x_k^0$, $m_k^0$ and $\pi_k$, $k\in[d]$, be as above, $g$ be defined by~\eqref{f_function_g} and $X$ be an $L_2^{\uparrow}[0,b]$-valued continuous square integrable martingale with the quadratic variation $\int_0^{\cdot}\pr_{X(s)}ds$. Then there exists a system of processes $\{x_k(t),\ t\in[0,T],\ k\in[d]\}$ satisfying $(F1)-(F4)$ such that for all $t\in[0,T]$
 $$
 X(t)=\sum_{k=1}^dx_k(t)\I_{\pi_k}\quad\mbox{a.s.}
 $$
\end{lem}
\begin{proof}
By Proposition~\ref{prp_coalescing} and Remark~\ref{rem_size_of_Pi}, $\p\{|\Pi_{X(t)}|\leq|\Pi_g|=d,\ t\in[0,T]\}=1$. So, $X(t)\in\St$ a.s. for all $t\in[0,T]$. Hence, by Proposition~\ref{prp_modif_from_D_for_mart}, there exist a modification of $X$ from $D((0,b),C[0,T])$. We denote the modification also by $X$.

Let $\delta=\min\limits_{k\in[d]}|\pi_k|$ and $h_k=\frac{1}{\delta}\I_{[a_k,a_k+\delta]},$ where $a_0=0$, $a_k=a_{k-1}+m_k^0$, $k\in[d]$. Set
$$
x_k(t)=(X(t),h_k),\quad t\in[0,T],\ \ k\in[d].
$$
Then by Proposition~\ref{prp_modif_from_D_for_mart}, the system $\{x_k(t),\ t\in[0,T],\ k\in[d]\}$ satisfies $(F1)-(F3)$ and by~\eqref{f_coalescing}, for all $t\in[0,T]$
$$
X(t)=\sum_{k=1}^dx_k(t)\I_{\pi_k}\quad\mbox{a.s.}
$$
We evaluate
\begin{align*}
 \langle x_k(\cdot),x_l(\cdot)\rangle_t&=\left\langle(X(\cdot),h_k),(X(\cdot),h_l)\right\rangle_t\\
 &=\int_0^t(\pr_{X(s)}h_k,\pr_{X(s)}h_l)ds=\int_0^t\frac{\I_{\{\tau_{k,l}\leq s\}}ds}{m_k(s)},\quad t\in[0,T].
\end{align*}
It finishes the proof.
\end{proof}

Lemmas~\ref{lem_uniqueness_x}~and~\ref{lem_connection_between_X_x} immediately imply the following result.

\begin{prp}\label{prp_uniqueness}
For each $g\in\St$ there exists an $L_2^{\uparrow}[0,b]$-valued continuous square integrable martingale $X(t)$, $t\in[0,T]$, with the quadratic variation $\langle X\rangle_t=\int_0^t\pr_{X(s)}ds$ which starts from $g$. Moreover, if $Y(t)$, $t\in[0,T]$, is other $L_2^{\uparrow}[0,b]$-valued continuous square integrable martingale with the same quadratic variation that starts from $g$, then the distributions of $X$ and $Y$ coincide in $C([0,T],L_2^{\uparrow}[0,b])$.
\end{prp}

We denote the distribution of the $L_2^{\uparrow}$-valued continuous square integrable martingale with quadratic variation $(M')$ starting from $g$ in the space $C([0,T],L_2^{\uparrow})$ by $P_g$. We will consider the set of step functions $\St$ as a topological subspace of $L_2^{\uparrow}$ with the induced topology. Let $\Pp$ be the space of all probability measures on $C([0,T],L_2^{\uparrow})$, endowed with the weak topology. Since the system of processes $\{x_k(t),\ t\in[0,T],\ k\in[d]\}$ satisfying $(F1)-(F3)$ can be constructed by coalescence of Wiener trajectories (see e.g.~\cite{Konarovskiy:2010:TVP:en,Konarovskiy:2010:UMJ:en}) and $X$ can be defined by~\eqref{f_connection_X_x}, it is easy to see that the map $P_{\cdot}:\St\to\Pp$ is measurable. Consequently, the probability measures
$$
P^{\xi}=\int_{\St}P_g\Xi(dg)
$$
is well-defined for any random element $\xi$ in $\St$ with the distribution $\Xi$.

\begin{prp}\label{proposition_distribution_of_mart}
A process $X(t)$, $t\in[0,T]$, with $X(0)\in\St$ a.s., is an $L_2^{\uparrow}$-valued continuous square integrable martingale with the quadratic variation $(M')$ if and only if $\law\{X\}=P^{X(0)}$.
\end{prp}
\begin{proof}
 The statement follows from the existence of regular conditional distribution of $X$ given $\sigma(X(0))$ (see~Theorem~1.3.1~\cite{Watanabe:1981:en}) and Proposition~\ref{prp_uniqueness}.
\end{proof}

\section{The main estimations}\label{estimations}

In this section we will suppose that $Y(u,\cdot),\ u\in(a,b)$, is a $C[0,T]$-valued random process with trajectories in $D((a,b),C[0,T])$ that satisfies $(C1)-(C4)$ with $(0,b)$ replaced by $(a,b)$.

It should be noted that in this section we do not claim that the process $Y$ exists, here we only study properties of $Y$ if it exists.

We will interpret $Y$ as the description of the evolution of particles on the real line which coalesce and change their masses and diffusion rates. Since $m(u,t)$ is the mass of particle at time $t$ that starts from $g(u)$, the inequality $m(u,t)<r$ implies that the particles starting from $g(u)$ and $g(u+r)$ ($g(u-r)$) have not coalesced by $t$. Moreover, the particle, which starts from $g(u)$, has diffusion rate grater then $\frac{1}{r}$. Consequently, $\p\{m(u,t)<r\}$ can be estimated by $\p\{\mbox{the Wiener process starting from}\ g(u+r)-g(u)$ $\mbox{with diffusion}\ \frac{1}{r}\ \mbox{does not hit 0 by time}\ t\}$. This is the main idea of the proof of the following lemma that is the key statement that allows to prove the existence of a martingale with the quadratic variation $(M')$ which starts from $g\in L_p^{\uparrow}$ and to study its asymptotic behavior.

\begin{lem}\label{lemma_estimation_of_p_m}
 For all $u\in(a,b)$, $0<r<b-u$ and $t\in(0,T]$
$$
\p\{m(u,t)<r\}\leq\frac{2}{\sqrt{2\pi}}\int_0^{\frac{G(u,r)\sqrt{r}}{\sqrt{t}}}e^{-\frac{x^2}{2}}dx\leq \frac{2\sqrt{r}}{\sqrt{2\pi t}}G(u,r),
$$
where $G(u,r)=g(u+r)-g(u)$.
\end{lem}

\begin{rem}\label{rem_about_G}
 The lemma also is true if the assumption $0<r<b-u$ is replaced by $0<r<u-a$ and the function $G(u,r)=g(u+r)-g(u)$ by $G(u,r)=g(u)-g(u-r)$.
\end{rem}

\begin{proof}[Proof of Lemma~\ref{lemma_estimation_of_p_m}]
The proof is similar to the proof of Lemma~2.16~\cite{Konarovskyi:2014:arx}. Let $0<r<b-u$. We denote
$$
M(t)=Y(u+r,t)-Y(u,t)
$$
and
$$
A_t=\{m(u,t)<r\}.
$$
Note that $M(\cdot)$ is a continuous square integrable martingale with the quadratic variation
$$
\langle M(\cdot)\rangle_t=\left\langle Y(u+r,\cdot)\right\rangle_t+\langle Y(u,\cdot)\rangle_t- 2\left\langle Y(u+r,\cdot),Y(u,\cdot)\right\rangle_t.
$$
By $(C4)$, we have
$$
\left\langle Y(u+r,\cdot),Y(u,\cdot)\right\rangle_t\I_{\{M(t)>0\}}=0,\quad t\in[0,T].
$$
Taking $\omega\in A_t$, we see that $\omega\in\{M(t)>0\}$ because $Y(u+r,\cdot,\omega)$ and $Y(u,\cdot,\omega)$ do not meet by time $t$. Hence,
\begin{align*}
\langle M(\cdot)\rangle_t(\omega)&=\left\langle Y(u+r,\cdot)\right\rangle_t(\omega)+\langle Y(u,\cdot)\rangle_t(\omega)\\
&\geq\langle Y(u,\cdot)\rangle_t(\omega)=\int_0^t\frac{ds}{m(u,s,\omega)}\geq \frac{t}{r}.
\end{align*}
Next, since $M(\cdot)$ is a continuous square integrable martingale, there exists a Wiener process $w(t),\ t\geq 0$, such that
\begin{equation}\label{f_M_w}
M(t)=G(u,r)+w\left(\langle M(\cdot)\rangle_t\right),
\end{equation}
by Theorem~2.7.2'~\cite{Watanabe:1981:en}. We set
$$
\tau=\inf\{t:\ M(t)=0\}\wedge T\ \ \mbox{and}\ \ \sigma=\inf\{t:\ G(u,r)+w(t)=0\}.
$$
It is easy to see that~\eqref{f_M_w} implies
$$
\langle M(\cdot)\rangle_{\tau}\leq\sigma.
$$
Note that if $\omega\in A_t$, then $\tau(\omega)>t$ and hence, by the last inequality,
$$
\sigma(\omega)\geq\langle M(\cdot)\rangle_{\tau(\omega)}(\omega)\geq \langle M(\cdot)\rangle_{t}(\omega)\geq \frac{t}{r}.
$$
Now we are ready to estimate the probability of $A_t$. So,
\begin{align*}
\p\{A_t\}&=\p\{A_t,\ M(t)>0\}=\p\{A_t,\ \tau>t\}\leq\p\left\{A_t,\ \sigma>\frac{t}{r}\right\}\\
&\leq\p\left\{\sigma>\frac{t}{r}\right\}=\p\left\{\max\limits_{s\in[0,t/r]}w(s)<G(u,r)\right\}\leq \p\left\{\max\limits_{s\in[0,1]}w(s)<\frac{G(u,r)\sqrt{r}}{\sqrt{t}}\right\}\\
&\leq\frac{2}{\sqrt{2\pi}}\int_0^{\frac{G(u,r)\sqrt{r}}{\sqrt{t}}}e^{-\frac{x^2}{2}}dx\leq\frac{2}{\sqrt{2\pi}}\frac{G(u,r)\sqrt{r}}{\sqrt{t}}.
\end{align*}
It finishes the proof.
\end{proof}

\begin{prp}\label{proposition_estimation_of_m}
Let $p>1$. Then for every $g\in L_p^{\uparrow}[a,b]$ and $0<\beta<\frac{3}{2}-\frac{1}{p}$
$$
\E\int_a^b\frac{du}{m^{\beta}(u,t)}\leq \frac{C_{p,\beta,a,b}}{\sqrt{t}}(1+\|g\|_{L_p}), \quad t\in(0,T].
$$
\end{prp}

\begin{proof}
 Without loss of generality, we assume that $a=0$. Using Lemma~\ref{lemma_estimation_of_p_m} and H\"{o}lder's inequality, we can estimate
\begin{align*}
\int_0^{\frac{b}{2}}&\E\frac{1}{m^{\beta}(u,t)}du=\int_0^{\frac{b}{2}}\int_0^{\infty}\p\left\{\frac{1}{m(u,t)}>r^{\frac{1}{\beta}}\right\}dudr\\
&\leq\int_0^{\frac{b}{2}}\int_0^{\frac{2^{\beta}}{b^{\beta}}}1dudr+\frac{2}{\sqrt{2\pi t}}\int_0^{\frac{b}{2}}\int_{\frac{2^{\beta}}{b^{\beta}}}^{\infty}\frac{g\left(u+\frac{1}{r^{1/\beta}}\right)-g(u)}{r^{\frac{1}{2\beta}}}dudr\\
&=\frac{2^{\beta-1}}{b^{\beta-1}}+\frac{2}{\sqrt{2\pi t}}\int_{\frac{2^{\beta}}{b^{\beta}}}^{\infty}\left[\frac{1}{r^{\frac{1}{2\beta}}}\int_0^b\left(\I_{\left[\frac{1}{r^{1/\beta}},\frac{b}{2}+\frac{1}{r^{1/\beta}}\right]}(u)-\I_{\left[0,\frac{b}{2}\right]}(u)\right)g(u)du\right]dr\\
&\leq\frac{2^{\beta-1}}{b^{\beta-1}}+\frac{2}{\sqrt{2\pi t}}\int_{\frac{2^{\beta}}{b^{\beta}}}^{\infty}\left[\frac{1}{r^{\frac{1}{2\beta}}}\left\|\I_{\left[\frac{1}{r^{1/\beta}},\frac{b}{2}+\frac{1}{r^{1/\beta}}\right]}-\I_{\left[0,\frac{b}{2}\right]}\right\|_{L_q}\|g\|_{L_p}\right]dr\\
&\leq\frac{2^{\beta-1}}{b^{\beta-1}}+\frac{2^{\frac{1}{q}+1}}{\sqrt{2\pi t}}\int_{\frac{2^{\beta}}{b^{\beta}}}^{\infty}\frac{1}{r^{\frac{1}{2\beta}+\frac{1}{\beta q}}}dr\|g\|_{L_p}\leq \frac{C_{p,\beta,b}}{\sqrt{t}}(1+\|g\|_{L_p}),
\end{align*}
where $\frac{1}{p}+\frac{1}{q}=1$.

Similarly, using Lemma~\ref{lemma_estimation_of_p_m} and Remark~\ref{rem_about_G}, we obtain
$$
\int_{\frac{b}{2}}^b\E\frac{1}{m^{\beta}(u,t)}du\leq\frac{C_{p,\beta,b}}{\sqrt{t}}(1+\|g\|_{L_p}).
$$
The proposition is proved.
\end{proof}

Next, let $X(t)$, $t\in[0,T]$, be an $L_2^{\uparrow}[0,b]$-valued continuous square integrable martingale with the quadratic variation $\langle X\rangle_t=\int_0^t\pr_{X(s)}ds$ which starts from $g\in\St$. By Proposition~\ref{lem_connection_between_X_x}, $X$ has a modification from $D((0,b),C[0,T])$ that satisfies $(C1)-(C4)$. Consequently, Proposition~\ref{proposition_estimation_of_m} is applicable to $X$.

\begin{prp}\label{proposition_estimation_of_z}
For every $g\in\St$, $0\leq\delta<1$ and $\eps>\frac{2\delta}{1-\delta}$
$$
\E\sup_{s\in[0,t]}\|X(s)-g\|_{L_{2+\delta}}^{2+\delta}\leq C_{\delta,\eps,b}t^{1+\frac{\delta}{2}}\left(1+\|g\|_{L_{2+\eps}}\right), \quad t\in[0,T].
$$
\end{prp}

\begin{proof}
Without loss of generality, let $X$ be defined by~\eqref{f_connection_X_x}. Using the Burkholder-Davis-Gundy inequality and Proposition~\ref{proposition_estimation_of_m}, we obtain
\begin{align*}
\E\sup_{s\in[0,t]}\int_0^b|X(u,s)&-g(u)|^{2+\delta}du=\E\sup_{s\in[0,t]}\left(\sum_{k=1}^dm_k^0|x_k(s)-x_k^0|^{2+\delta}\right)\\
&\leq\sum_{k=1}^dm_k^0\E\sup_{s\in[0,t]}|x_k(s)-x_k^0|^{2+\delta}\leq
\sum_{k=1}^dm_k^0\E\left(\int_0^t\frac{ds}{m_k(s)}\right)^{1+\frac{\delta}{2}}\\
&\leq t^{\frac{\delta}{2}}\sum_{k=1}^dm_k^0\E\int_0^t\frac{ds}{m_k^{1+\frac{\delta}{2}}(s)}=
t^{\frac{\delta}{2}}\int_0^t\left(\E\int_0^b\frac{du}{m^{1+\frac{\delta}{2}}(u,s)}\right)ds\\
&\leq C_{\delta,\eps,b}t^{1+\frac{\delta}{2}}\left(1+\|g\|_{L_{2+\eps}}\right),
\end{align*}
if $1+\frac{\delta}{2}<\frac{3}{2}-\frac{1}{2+\eps}$. The proposition is proved.
\end{proof}

\begin{cor}\label{corollary_estimation_of_z}
 Under the assumptions of Proposition~\ref{proposition_estimation_of_z},
$$
\E\sup_{t\in[0,T]}\|X(t)\|_{L_{2+\delta}}^{2+\delta}\leq C_{\delta,\eps,b}\left(1+\|g\|_{L_{2+\delta}}^{2+\delta}+\|g\|_{L_{2+\eps}}\right).
$$
\end{cor}

\begin{rem}\label{remark_estimation_of_Pi_y}
Let $X$ be as in Proposition~\ref{proposition_estimation_of_z}. It is easily seen that
$$
|\Pi_{X(t)}|=\int_0^1\frac{du}{m(u,t)}
$$
and consequently, by Proposition~\ref{proposition_estimation_of_m}, for all $\eps>0$
$$
\E|\Pi_{X(t)}|\leq\frac{C_{\eps}}{\sqrt{t}}\left(1+\|g\|_{L_{2+\eps}}\right), \quad t\in(0,T].
$$
\end{rem}

\section{Construction of the particle system in general case of initial mass distribution}\label{construction}

\subsection{The tightness of $\{P_{g_{\alpha}},\ \alpha\in I\}$}

In this subsection we show that the family of distributions $\{P_{g_{\alpha}},\ \alpha\in I\}$ is tight under the assumption $\{\|g_{\alpha}\|_{L_{2+\eps}}\}$ is bounded for some $\eps>0$, where $I$ is a set of indices. First we construct suitable compacts in $L_2^{\uparrow}$.

\begin{lem}\label{lemma_compact_in_L2}
 For every $M>0$ and $\delta>0$ the set $K_M=\{g\in L_2^{\uparrow}:\ \|g\|_{L_{2+\delta}}\leq M\}$ is compact in $L_2^{\uparrow}$.
\end{lem}

\begin{proof}
Let $\{g_n\}_{n\geq 1}$ be a sequence in $K_M$. Since $\{g_n\}_{n\geq 1}\subset L_2^{\uparrow}$ is bounded, one can choose a subsequence $\{g_{n'}\}$ that converges a.e. to a nondecreasing function $g$, possible taking infinite values. Next, by the de la Vallee-Poussin theorem (see Theorem~1.8.~\cite{Liptser:2001}), $\{g_{n'}^2\}$ is uniformly integrable. Therefore, $\|g_{n'}^2\|_{L_1}=\|g_{n'}\|_{L_2}^2\to \|g\|_{L_2}^2$. It implies the convergence of $\{g_{n'}\}$ to $g$ in $L_2$, by Lemma~1.32~\cite{Kallenberg:2002}. This finishes the proof.
\end{proof}

\begin{prp}\label{proposition_compactness}
 Let $\{g_n,\ n\geq 1\}\subset\St$ be bounded in $L_{2+\eps}$ for some $\eps>0$. Then the family of the distributions $\{P_{g_n},\ n\geq 1\}$ is tight in $\Pp$.
\end{prp}

\begin{proof}
Let $X_n(t)$, $t\in[0,T]$, $n\geq 1$, be random elements in $C([0,T],L_2^{\uparrow})$ with distributions $P_{g_n}$, $n\geq 1$. To prove the proposition, we will use Jakubowski's tightness criterion~\cite{Jakubowski:1986}. We will check that
\begin{enumerate}
 \item[(J1)] for every $\gamma>0$ there exists a compact $K\subset L_2^{\uparrow}$, such that
$$
\p\{\exists t\in[0,T],\ X_n(t)\notin K\}\leq\gamma,\quad n\geq 1;
$$

 \item[(J2)] for every $h\in L_2$ the sequence $\{(X_n(\cdot),h)\}_{n\geq 1}$ is
tight in $C([0,T],\R)$.
\end{enumerate}

Property $(J1)$ follows from Corollary~\ref{corollary_estimation_of_z}, Lemma~\ref{lemma_compact_in_L2} and Chebyshev's inequality. In fact, choosing $\delta>0$ such that $\frac{2\delta}{1-\delta}\vee\delta<\eps$ and setting $K_M=\{g\in L_2^{\uparrow}:\ \|g\|_{L_{2+\delta}}\leq M\}$, we have
\begin{align*}
\p\{\exists t\in[0,T],\ &X_n(t)\notin K_M\}=\p\left\{\sup_{t\in[0,T]}\|X_n(t)\|_{L_{2+\delta}}>M\right\}\\
&\leq\frac{1}{M^{2+\delta}}\E\sup_{t\in[0,T]}\|X_n(t)\|_{L_{2+\delta}}^{2+\delta}\\
&\leq \frac{C_{\delta,\eps}}{M^{2+\delta}}\left(1+\|g_n\|_{L_{2+\delta}}^{2+\delta}+\|g_n\|_{L_{2+\eps}}\right)\leq\gamma
\end{align*}
for large enough $M$ and all $n\geq 1$.

Since for all $h\in L_2$ the process $(X_n(\cdot),h)$ is a continuous square integrable martingale with the quadratic variation
$$
\langle(X_n(\cdot),h)\rangle_t=\int_0^t\|\pr_{X_n(s)}h\|^2_{L_2}ds,\quad t\in[0,T],
$$
and $\|\pr_{X_n(t)}h\|_{L_2}\leq\|h\|_{L_2}$, $t\in[0,T]$, the Aldous tightness criterion (see e.g. Theorem~3.6.4.~\cite{Dawson:1993}) easily implies $(J2)$. It completes the proof of the proposition.
\end{proof}

\subsection{Some limit properties}

In this subsection we show that under the assumption $\{|\Pi_{g_n}|,\ n\geq 1\}$ is bounded, each limit point of the set $\{P_{g_n}\}_{n\geq 1}$ is $P_{g}$ for some $g\in\St$.
\begin{lem}\label{lemma_property_of_linit_point_for_St}
 Let $X_n$, $n\geq 1$, be random elements in $C([0,T],L_2^{\uparrow}[0,b])$ with distributions $P_{g_n}$, where $g_n\in\St$, $n\geq 1$, and $\{|\Pi_{g_n}|,\ n\geq 1\}$ is bounded. If the sequence $\{X_n\}_{n\geq 1}$ converges to $X$ in distribution, then $\law\{X\}=P_{X(0)}$.
\end{lem}

\begin{proof}
Let $\Pi_{g_n}=\{\pi_1^n,\ldots,\pi_{q_n}^n\}$, where elements of $\pi_k^n$ are less then elements of $\pi_{k+1}^n$, $k\in[q_n-1]$. Since $\{q_n\}_{n\geq 1}$ is bounded, there exist an infinite sequence $\{n'\}$ and $q\in\N$ such that $q_{n'}=q$ for all $n'$. Without loss of generality, we may assume that $q_n=q$ for all $n\geq 1$. Next, setting $m_k^{0,n}=|\pi_k^n|$, $k\in[q]$, $n\geq 1$, and using the boundedness of $\{m_k^{0,n}\}_{n\geq 1}$, we can choose a sequence $\{n'\}$ such that for all $k\in[q]$
$$
m_k^{0,n'}\to m_k^0,
$$
where $m_k^0\in[0,b]$. Again, without loss of generality, we assume that $n'=n$.

Set
\begin{align*}
 I&=\{k\in[q]:\ m_k^0>0\},\\
 I^c&=[q]\setminus I
\end{align*}
and
$$
x_k^{0,n}=\frac{1}{m_k^{0,n}}\int_{\pi_k^n}g_n(u)du,\quad k\in[q],\ \ n\geq 1.
$$
Since $m_k^{0,n}\to m_k^0>0$, $k\in I$, and $\{\|g_n\|_{L_2}\}_{n\geq 1}$ is bounded, it is easy to see that $\{x_k^{0,n}\}_{n\geq 1}$ is also bounded for all $k\in I$. Thus, there exists a sequence $\{n'\}$ such that $x_k^{0,n'}\to x_k^0$ for all $k\in I$. Let again $n'=n$.

Next, let $I=\{k_1,\ldots,k_l\}$, where $k_i<k_{i+1}$, $i\in[l-1]$. We set $a_0=0$, $a_i=a_{i-1}+m_{k_i}$, $i\in[l]$, and $\pi'_i=[a_{i-1},a_i)$, $i\in[l-1]$, $\pi'_l=[a_{l-1},a_l]$. Since $g_n=X_n(0)\to X(0)=g$ in $L_2$, one can show that
$$
g=\sum_{i=1}^lx_{k_i}^0\I_{\pi_i'},
$$
and for all $\pi\in\Pi_g$ there exist $\pi_{i_1},\ldots,\pi_{i_{l'}}$ such that $\pi=\bigcup_{j=1}^{l'}\pi'_{i_j}$.

We set
\begin{align*}
h_k^n&=\frac{1}{m_k^{0,n}}\I_{\pi_k^n},\quad k\in[q],\ \ n\geq 1,\\
h_i&=\frac{1}{m_{k_i}^0}\I_{\pi'_i},\quad i\in[l].
\end{align*}
By the construction of $m_{k_i}$ and $\pi_i'$, $i\in[l]$,  we have $h_{k_i}^n\to h_i$ in $L_2$ for all $i\in[l]$.

Next, using Skorohod's theorem (see Theorem~3.1.8~\cite{Ethier:1986}), we may assume that
$$
X_n\to X\quad \mbox{a.s. in }\ \ C([0,T],L_2^{\uparrow}).
$$
Let
\begin{align*}
 x_k^n(t)&=(h_k^n,X_n(t)),\quad t\in[0,T],\ \ k\in[q],\ \ n\geq 1,\\
 x_i(t)&=(h_i,X(t)),\quad t\in[0,T],\ \ i\in[l].
\end{align*}
We note that, by Proposition~\ref{prp_coalescing}, for all $t\in[0,T]$
\begin{align*}
X_n(t)&=\sum_{k=1}^qx_k^n(t)\I_{\pi_i^n}\quad\mbox{a.s.},\ \ n\geq 1,\\
X(t)&=\sum_{i=1}^lx_i(t)\I_{\pi'_i}\quad\mbox{a.s.}
\end{align*}
It is easy to see that for all $i\in[l]$
$$
x_{k_i}^n\to x_i\quad \mbox{a.s. in }\ \ C([0,T],\R).
$$
Let us show that the family $\{x_i(t),\ t\in[0,T],\ i\in[l]\}$ satisfies $(F1)-(F4)$. First, we show that $x_i$ is a square integrable martingale with respect to the joint filtration $\bigcap_{\eps>0}\sigma(x_i(s), s\leq t+\eps,\ i\in[l])$, $t\in[0,T]$. But since each $x_i$ is continuous, it is enough to check that $x_i$ is a square integrable martingale with respect to $\F_t=\sigma(x_i(s), s\leq t,\ i\in[l])$, $t\in[0,T]$. Let $m_i^n(t)$, $\tau_{i,j}^n$, and $m_i(t)$, $\tau_{i,j}$ are defined as before for $\{x_k^n(t),\ t\in[0,T],\ k\in[q]\}$ and $\{x_i(t),\ t\in[0,T],\ i\in[l]\}$, respectively.

We can estimate the second moment of $x_{k_i}^n(t)$, $i\in[l]$, as follows
$$
\E(x_{k_i}^n(t)-x_{k_i}^{0,n})^2=\E\int_0^t\frac{ds}{m_{k_i}^n(s)}\leq \frac{t}{m_{k_i}^{0,n}}\leq C,
$$
where $C$ is a constant that is independent of $n$, $t$ and $k_i$. By Fatou's lemma $\E x_i^2(t)\leq C$ for all $t\in[0,T]$ and $i\in[l]$. Therefore, Proposition~9.1.17~\cite{Jacod:2003} implies that $x_i$ is a continuous $(\F_t)$-square integrable martingale for any $i\in[l]$. To finish the proof of the lemma, we show that the joint quadratic variation of $x_i$ and $x_j$ satisfies $(F4)$.

By Lemma~2.10~\cite{Konarovskyi:2014:arx}, for each $i,j\in[l]$, $\tau_{k_i,k_j}^n\to\tau_{i,j}$ in probability. Since we can choose a sequence $\{n'\}$ such that $\tau_{k_i,k_j}^{n'}\to\tau_{i,j}$ a.s. for all $i,j=1,\ldots,l$, without loss of generality, we may suppose that $\tau_{k_i,k_j}^{n}\to\tau_{i,j}$ a.s. Let us denote
$$
R=\{t\in[0,T]:\ \exists i,j\ \p\{\tau_{i,j}=t\}>0\}.
$$
It is easily seen that $\leb\{R\}=0$ and for all $t\in R^c=[0,T]\setminus R$
$$
\I_{\{\tau_{k_i,k_j}^n\leq t\}}\to\I_{\{\tau_{i,j}\leq t\}}\quad\mbox{a.s.},\ \ i,j\in[l].
$$
Note, that in
\begin{align*}
 m_{k_i}^n(t)=\sum_{k=1}^qm_k^{0,n}\I_{\{\tau_{k_i,k}^n\leq t\}}=\sum_{j=1}^lm_{k_j}^{0,n}\I_{\{\tau_{k_i,k_j}^n\leq t\}}+\sum_{k\in I^c}m_k^{0,n}\I_{\{\tau_{k_i,k}^n\leq t\}}
\end{align*}
the first term of the right hand side tends to $m_i(t)=\sum_{j=1}^lm_{k_j}^0\I_{\{\tau_{i,j}\leq t\}}$ a.s. and the second term tends to zero. So, $m_{k_i}^n(t)\to m_i(t)$ a.s. for all $i=1,\ldots,l$ and $t\in R^c$. Since $\frac{1}{m_{k_i}^n(t)}\leq \frac{1}{m_{k_i}^{0,n}}\to\frac{1}{m_{k_i}^0}<\infty$, the sequence $\left\{\frac{1}{m_{k_i}^n(t)}\right\}_{n\geq 1}$ is bounded uniformly by $t$ for all $i\in[l]$. Hence, by the dominated convergence theorem, we obtain
$$
\langle x_{k_i}^n,x_{k_j}^n\rangle_t=\int_0^t\frac{\I_{\{\tau_{k_i,k_j}^n\leq s\}}}{m_{k_i}^n(s)}ds\to \int_0^t\frac{\I_{\{\tau_{i,j}\leq s\}}}{m_i(s)}ds\quad\mbox{a.s.}\quad\mbox{for all}\ \ i,j\in[l].
$$
Thus, Lemma~B.11.~\cite{Engelbert:2005} implies that
$$
\langle x_i,x_j\rangle_t=\int_0^t\frac{\I_{\{\tau_{i,j}\leq s\}}}{m_i(s)}ds.
$$
The lemma is proved.
\end{proof}

\begin{prp}\label{proposition_converg_result}
 Let $\{g_n\}_{n\geq 1}\subset\St$, $g_n\to g$ in $L_2$ and for some $\eps>0$ the sequences $\{\|g_n\|_{L_{2+\eps}}\}_{n\geq 1}$, $\{|\Pi_{g_n}|\}_{n\geq 1}$ be bounded. Then $P_{g_n}\to P_g$ in $\Pp$.
\end{prp}

\begin{proof}
 By Proposition~\ref{proposition_compactness} and Lemma~\ref{lemma_property_of_linit_point_for_St}, every subsequence of $\{P_{g_n}\}_{n\geq 1}$ has a subsubsequence converging to $P_g$. It proves the proposition.
\end{proof}

\subsection{Existence in the general case}

In this section we construct an $L_2^{\uparrow}$-valued continuous square integrable martingale with the quadratic variation $(M')$ starting from $g\in L_{2+\eps}^{\uparrow}$ as a weak limit of processes with distributions $P_{g_n}$, $g_n\in\St$.

\begin{thm}\label{theorem_existence}
 Let $\eps>0$. Then for every $g\in L_{2+\eps}^{\uparrow}[0,b]$ there exists an $L_2^{\uparrow}[0,b]$-valued continuous square integrable martingale $X(t)$, $t\in[0,T]$, with the quadratic variation $(M')$ that starts from $g$.
\end{thm}

\begin{proof}
We set $\s_n=\sigma\left(\left[\frac{k-1}{2^n},\frac{k}{2^n}\right),\ k\in[2^n]\right)$, $n\geq 1$, and
$g_n=\E_{\leb}(g|\s_n)$, where $\E_{\leb}$ denotes the conditional expectation on the probability space $([0,1],\B([0,1]),\leb)$. Since $g_n\to g$ in $L_{2+\eps}$ (see~\cite{Alonso:1998}), the sequence $\{\|g_n\|_{L_{2+\eps}}\}_{n\geq 1}$ is bounded. Therefore the sequence $\{P_{g_n}\}_{n\geq 1}$ is tight in $\Pp$, by Proposition~\ref{proposition_compactness}.

Let $X_n(t)$, $t\in[0,T]$, have distribution $P_{g_n}$ for each $n\geq 1$. Corollary~\ref{corollary_estimation_of_z} and Remark~\ref{remark_estimation_of_Pi_y} imply that for all $r\in(0,T]\cap\Q$ and some fixed $\delta<\frac{\eps}{2+\eps}<\eps$, $\{|\Pi_{X_n(r)}|\}_{n\geq 1}$ and $\{\|X_n(r)\|_{L_{2+\delta}}\}_{n\geq 1}$ are tight in $\R$. Thus, by Proposition~3.2.4.~\cite{Ethier:1986}, the sequence of the random vectors $\{(X_n,(|\Pi_{X_n(r_k)}|)_{k\in\N},(\|X_n(r_k)\|_{L_{2+\delta}})_{k\in\N})\}_{n\geq 1}$ is tight in $C([0,T],L_2^{\uparrow})\times\R^{\N}\times\R^{\N}$, where $\{r_k,\ k\in\N\}=(0,T]\cap\Q$. So, Prokhorov's theorem~\cite{Billingsley:1999} yields that the sequence $\{(X_n,(|\Pi_{X_n(r_k)}|)_{k\in\N},(\|X_n(r_k)\|_{L_{2+\delta}})_{k\in\N})\}_{n\geq 1}$ is relatively compact and consequently, there exists a subsequence $\{n'\}$ such that
$$
(X_{n'},(|\Pi_{X_{n'}(r_k)}|)_{k\in\N},(\|X_{n'}(r_k)\|_{L_{2+\delta}})_{k\in\N})\to (X,(\xi_k)_{k\in\N},(\eta_k)_{k\in\N})
$$
in $C([0,T],L_2^{\uparrow})\times\R^{\N}\times\R^{\N}$ in distribution. For convenience of notation, we suppose that $n'=n$. Next, by Skorohod's theorem (see Theorem~3.1.8~\cite{Ethier:1986}), we may assume that
$$
(X_n,(|\Pi_{X_n(r_k)}|)_{k\in\N},(\|X_n(r_k)\|_{L_{2+\delta}})_{k\in\N})\to (X,(\xi_k)_{k\in\N},(\eta_k)_{k\in\N})\quad\mbox{a.s.}
$$

Since $\{|\Pi_{X_n(r_k)}|\}_{n\geq 1}$ and $\{\|X_n(r_k)\|_{L_{2+\delta}}\}_{n\geq 1}$ are convergent a.s., they are bounded a.s. Thus, the event
$$
\Omega'=\{X_n\to X,\ \{|\Pi_{X_n(r_k)}|\}_{n\geq 1}\ \mbox{and}\ \{\|X_n(r_k)\|_{L_{2+\delta}}\}_{n\geq 1}\ \ \mbox{are bounded},\ k\geq 1\}
$$
has probability 1. It is easy to see that for all $k\geq 1$, $X_n(r_k+t)$, $t\in[0,T-r_k]$, is an $L_2^{\uparrow}$-valued continuous square integrable martingale with the quadratic variation $(M')$. Hence, Proposition~\ref{proposition_distribution_of_mart} implies that
$$
\law\{X_n(r_k+\cdot)\}=P^{X_n(r_k)}\quad\mbox{for all}\ \ n,k\geq 1.
$$
Since for all $\omega\in\Omega'$ the sequence $\{X_n(r_k,\omega)\}_{n\geq 1}$ converges to $X(r_k,\omega)$ and $\{|\Pi_{X_n(r_k,\omega)}|\}_{n\geq 1}$, $\{\|X_n(r_k,\omega)\|_{L_{2+\delta}}\}_{n\geq 1}$ are bounded, using Proposition~\ref{proposition_converg_result} and the dominated convergence theorem, we have
$$
P^{X_n(r_k)}\to P^{X(r_k)}\quad\mbox{as}\ \ n\to\infty.
$$
On the other hand, $P^{X_n(r_k)}\to\law\{X(r_k+\cdot)\}$ and consequently, $X(r_k+\cdot)$ has the distribution $P^{X(r_k)}$. So, from Proposition~\ref{proposition_distribution_of_mart} it follows that $X(r_k+\cdot)$ is an $L_2^{\uparrow}$-valued continuous square integrable martingale with the quadratic variation
$$
\langle(X(r_k+\cdot),h)\rangle_t=\int_0^t\|\pr_{X(r_k+s)}h\|_{L_2}^2ds=\int_{r_k}^{r_k+t}\|\pr_{X(s)}h\|_{L_2}^2ds
$$
for all $h\in L_2$. Since for each $h\in L_2$, $(X_n(\cdot),h)\to(X(\cdot),h)$ a.s. and
$$
\E((X_n(t),h)-(g_n,h))^2=\int_0^t\|\pr_{X_n(s)}h\|_{L_2}^2ds\leq t\|h\|_{L_2}^2,\quad t\in[0,T],
$$
one can show that $(X(\cdot),h)$ is a continuous square integrable martingale and
$$
\langle(X(r_k+\cdot),h)\rangle_t=\langle(X(\cdot),h)\rangle_{r_k+t}-\langle(X(\cdot),h)\rangle_{r_k}.
$$
Therefore,
$$
\langle(X(\cdot),h)\rangle_{r_k+t}=\langle(X(\cdot),h)\rangle_{r_k}+\int_{r_k}^{r_k+t}\|\pr_{X(s)}h\|_{L_2}^2ds.
$$ Making $r_{k'}\to 0$, we obtain
$$
\langle(X(\cdot),h)\rangle_t=\int_0^t\|\pr_{X(s)}h\|_{L_2}^2ds.
$$
The theorem is proved.
\end{proof}

\begin{rem}
 Let $X(t)$, $t\in[0,T]$, be the process constructed in the proof of Theorem~\ref{theorem_existence} with $g\in L_{2+\eps}$ for some $\eps>0$. Then Fatou's lemma and Proposition~\ref{proposition_estimation_of_z} implies that for each $0\leq\delta<\frac{\eps}{2+\eps}$
$$
\E\sup_{s\in[0,t]}\|X(s)-g\|_{L_{2+\delta}}^{2+\delta}\leq C_{\delta,\eps,b}t^{1+\frac{\delta}{2}}\left(1+\|g\|_{L_{2+\eps}}\right), \quad t\in[0,T],
$$
and consequently,
$$
\E\sup_{t\in[0,T]}\|X(t)\|_{L_{2+\delta}}^{2+\delta}\leq C_{\delta,\eps,b}\left(1+\|g\|_{L_{2+\delta}}^{2+\delta}+\|g\|_{L_{2+\eps}}\right).
$$
\end{rem}

\section{A modification in $D((0,b),C[0,T])$}\label{modification_from_D}

\subsection{Coalescence in a finite number of points}

We will prove that any $L_2^{\uparrow}$-valued continuous square integrable martingale $X(t)$, $t\in[0,T]$, with the quadratic variation $(M')$ takes values from $\St$ and for all $t\in(0,T]$
$$
\E|\Pi_{X(t)}|<\infty.
$$

Let us prove an auxiliary lemma.

\begin{lem}\label{lemma_prop_of_pr}
 Let $\{e_n\}_{n\geq 1}$ be an orthonormal basis of $L_2[0,b]$ and $g\in L_2^{\uparrow}[0,b]$. Then
\begin{equation}\label{f_fin_of_sum}
\sum_{n=1}^{\infty}\|\pr_ge_n\|_{L_2}^2<\infty
\end{equation}
if and only if $g\in\St$. Moreover,
$$
\sum_{n=1}^{\infty}\|\pr_ge_n\|_{L_2}^2=|\Pi_g|.
$$
\end{lem}

\begin{proof}
 We suppose that $g\in\St$ and prove~\eqref{f_fin_of_sum}. Let $\Pi_g=\{\pi_k,\ k\in[q]\}$. Then
$$
\pr_ge_n=\sum_{k=1}^q\frac{1}{|\pi_k|}\int_{\pi_k}e_n(u)du\I_{\pi_k}
$$
and consequently,
$$
\|\pr_ge_n\|_{L_2}^2=\sum_{k=1}^q\frac{1}{|\pi_k|}(e_n,\I_{\pi_k})^2.
$$
Hence,
\begin{align*}
\sum_{n=1}^{\infty}\|\pr_ge_n\|_{L_2}^2=\sum_{n=1}^{\infty}\sum_{k=1}^q\frac{1}{|\pi_k|}(e_n,\I_{\pi_k})^2=\sum_{k=1}^q\frac{1}{|\pi_k|}\|\I_{\pi_k}\|_{L_2}^2=q=|\Pi_g|<\infty.
\end{align*}

Next, suppose that~\eqref{f_fin_of_sum} holds. Then $\pr_g$ is a Hilbert-Schmidt operator. Since $\pr_g$ is a projection on the subspace of $\sigma(g)$-measurable functions $H_g\subset L_2$, it is easy to see that $H_g$ is a finite dimensional Hilbert space. Therefore $\sigma(g)$ is generated by a finite number of sets. This implies that $g\in\St$. The lemma is proved.
\end{proof}

\begin{prp}\label{prp_values_in_St}
Let $X(t)$, $t\in[0,T]$, be an $L_2^{\uparrow}[0,b]$-valued continuous square integrable martingale with the quadratic variation  $(M')$ which starts from $g\in L_{2+\eps}^{\uparrow}[0,b]$. Then
$$
\p\left\{\mbox{for all}\ 0<s\leq t\leq T,\ \ X(s)\in\St\ \mbox{and}\ \Pi_{X(t)}\leq \Pi_{X(s)}\right\}=1
$$
and
$$
\E\int_0^t|\Pi_{X(s)}|ds=\E\|X(t)-g\|_{L_2}^2<\infty.
$$
\end{prp}

\begin{proof}
Let $\{e_n\}_{n\geq 1}$ be an orthonormal basis of $L_2[0,b]$. Let us evaluate
\begin{align}\label{f_relation_EX2_Pi}
\begin{split}
\E\|X(t)-g\|_{L_2}^2&=\E\sum_{n=1}^{\infty}(X(t)-g,e_n)^2=\sum_{n=1}^{\infty}\E\langle(X(\cdot),e_n)\rangle_t\\
&=\sum_{n=1}^{\infty}\E\int_0^t\|\pr_{X(s)}e_n\|_{L_2}^2ds\\
&=\E\int_0^t\sum_{n=1}^{\infty}\|\pr_{X(s)}e_n\|_{L_2}^2ds=\E\int_0^t|\Pi_{X(s)}|ds<\infty.
\end{split}
\end{align}
We set
$$
\Omega'=\left\{\mbox{for all } 0<s\leq t\leq T,\ \Pi_{X(t)}\leq\Pi_{X(s)} \right\}
$$
and
$$
A=\left\{(s,\omega)\in(0,T]\times\Omega':\ \sum_{n=1}^{\infty}\|\pr_{X(s,\omega)}e_n\|_{L_2}^2<\infty\right\}.
$$
Then $\E\int_0^T\I_{A^c}(s)ds=0$, by Proposition~\ref{prp_coalescing} and~\eqref{f_relation_EX2_Pi}. Thus, there exists $\Omega''\subset\Omega'$ such that $\p\{\Omega''\}=1$ and $\leb([0,T]\setminus A_{\omega})=0$ for all $\omega\in\Omega''$, where $A_{\omega}=\{s\in(0,T]:\ (s,\omega)\in A\}$. It implies that $A_{\omega}$ is dense in $(0,T]$. Next, by Lemma~\ref{lemma_prop_of_pr}, we have that for each $\omega\in\Omega'$, $X(s,\omega)\in\St$, $s\in A_{\omega}$. Since for all $s\leq t$, $\omega\in\Omega''$, $\Pi_{X(t,\omega)}\leq \Pi_{X(s,\omega)}$ and $A_{\omega}$ is dense in $(0,T]$, we have that $X(s,\omega)\in\St$ for all $s\in (0,T]$ and $\omega\in\Omega''$. The proposition is proved.
\end{proof}

\begin{cor}
 For each $t\in(0,T]$, $\E|\Pi_{X(t)}|<\infty.$
\end{cor}

\begin{proof}
 By Proposition~\ref{prp_coalescing}, $\Pi_{X(t)}\leq \Pi_{X(s)}$ a.s. for all $s\leq t$. Thus, by Remark~\ref{rem_size_of_Pi}, $\E|\Pi_{X(t)}|\leq \E|\Pi_{X(s)}|$ for all $s\leq t$. So, the function $\E|\Pi_{X(t)}|$, $t\in(0,T]$, increases. Since the integral $\int_0^T\E|\Pi_{X(t)}|dt$ is finite, $\E|\Pi_{X(t)}|<\infty$ for all $t\in(0,T]$.
\end{proof}

\subsection{The martingale $X$ satisfies $(C1)-(C4)$ (proof of Theorem~\ref{theorem_modification})}

The aim of this section is to prove the existence of a process from $D((0,b),C[0,T])$ that satisfies $(C1)-(C4)$. We are going to show more, namely we prove that any $L_2^{\uparrow}[0,b]$-valued continuous square integrable martingale with the quadratic variation $(M')$ which starts from $g\in L_{2+\eps}^{\uparrow}[0,b]$ has a needed modification. It will prove Theorem~\ref{theorem_modification}.

So, let $X(t)$, $t\in[0,T]$, be an $L_2^{\uparrow}[0,b]$-valued continuous square integrable martingale with the quadratic variation $(M')$ which starts from $g\in L_{2+\eps}^{\uparrow}[0,b]$. Since the martingale $X(t)$, $t\in[0,T]$, takes values from $\St$ (see Proposition~\ref{prp_values_in_St}), Proposition~\ref{prp_modif_from_D_for_mart} implies that $X$ has a modification from $D((0,b),C[0,T])$. We will denote the modification of $X$ by the same letter $X$.

\begin{thm}\label{thm_conditions_C}
  The process $X(u,t),\ u\in(0,b),\ t\in[0,T]$, satisfies $(C1)-(C4)$.
\end{thm}

\begin{proof}
 Properties $(C1)-(C3)$ immediately follow from Proposition~\ref{prp_modif_from_D_for_mart}. Let us prove $(C4)$. We fix $u,v\in(0,b)$ and denote for $\eps>0$
 \begin{align*}
  h_{\eps}^u&=\frac{1}{\eps}\I_{[u,u+\eps]},\\
  h_{\eps}^v&=\frac{1}{\eps}\I_{[v,v+\eps]}.
 \end{align*}

 First we prove that for each $\lambda>0$ and $w=u,v$
 \begin{equation}\label{f_Q_convergence}
 \p\left\{\sup_{t\in[0,T]}|(X(t),h_{\eps}^w)-X(w,t)|>\lambda\right\}\to 0\quad\mbox{as}\ \eps\to 0.
 \end{equation}
 By Proposition~2.2.16~\cite{Ethier:1986},
 \begin{align*}
  \p\left\{\sup_{t\in[0,T]}|(X(t),h_{\eps}^w)-X(w,t)|>\lambda\right\}\leq\frac{1}{\lambda^2}\E((X(T),h_{\eps}^w)-X(w,T))^2.
 \end{align*}
 Since $X(T)\in D^{\uparrow}$, we have $(X(T),h_{\eps}^w)\to X(w,T)$ a.s. as $\eps\to 0$. Moreover,
 $$
 \p\left\{0\leq(X(T),h_{\eps}^w)-X(w,T)\leq X(w+\delta,T)-X(w,T)\ \mbox{for all}\ \eps<\delta\right\}=1,
 $$
 where $\delta$ is chosen such that $w+\delta\in(0,b)$. Since $\E(X(w+\delta,T)-X(w,T))^2<\infty$, the dominated convergence theorem implies
 $$
 \E((X(T),h_{\eps}^w)-X(w,T))^2\to 0\quad\mbox{as}\ \ \eps\to 0.
 $$
 It yields~\eqref{f_Q_convergence}. By Lemma~B.11~\cite{Engelbert:2005} and the polarization formula for joint quadratic variation of martingales,
 $$
 \sup_{t\in[0,T]}|\langle(X(\cdot),h_{\eps}^u),(X(\cdot),h_{\eps}^v)\rangle_t-\langle X(u,\cdot),X(v,\cdot)\rangle_t|\to 0\quad\mbox{in probability as}\ \ \eps\to 0.
 $$
 In particular, for all $t\in[0,T]$
 \begin{equation}\label{f_chat_lim}
 \langle(X(\cdot),h_{\eps}^u),(X(\cdot),h_{\eps}^v)\rangle_t\to\langle X(u,\cdot),X(v,\cdot)\rangle_t\quad\mbox{in probability as}\ \ \eps\to 0.
 \end{equation}

 Let $f\in\St$. Choose $\eps_0>0$ such that $[u,u+\eps_0]\in\pi$ and $[v,v+\eps_0]\in\pi'$ for $\pi,\pi'\in\Pi_f$. Then by~\eqref{f_projection}, for all $\eps\in(0,\eps_0]$
 \begin{equation}\label{f_lim_pr}
 (\pr_fh_{\eps}^u,\pr_fh_{\eps}^v)=\begin{cases}
                                    0,& \pi\neq\pi',\\
                                    \frac{1}{|\pi|},& \pi=\pi',
                                   \end{cases}
 =\frac{1}{|\pi|}\I_{\{f(u)=f(v)\}}.
 \end{equation}

 We next set
 \begin{align*}
 \Omega'&=\{X\in D((0,b),C[0,T])\}\\
 &\cap\left\{\mbox{for all}\ 0<s\leq t\leq T,\ \ X(s)\in\St\ \mbox{and}\ \Pi_{X(t)}\leq \Pi_{X(s)}\right\}.
 \end{align*}
 By propositions~\ref{prp_modif_from_D_for_mart} and~\ref{prp_values_in_St}, $\p\{\Omega'\}=1$.

 Let $\omega\in\Omega'$ and $\delta>0$. Since $\Pi_{X(s,\omega)}\leq \Pi_{X(\delta,\omega)}$ for all $s\in[\delta,T]$,~\eqref{f_lim_pr} implies that there exists $\eps_0(\omega)>0$ such that for each $\eps\in(0,\eps_0(\omega)]$ and $s\in[\delta,T]$
 $$
 (\pr_{X(s,\omega)}h_{\eps}^u,\pr_{X(s,\omega)}h_{\eps}^v)=\frac{1}{m(u,s,\omega)}\I_{\{X(u,s,\omega)=X(v,s,\omega)\}}=\frac{\I_{\{\tau_{u,v}(\omega)\leq s\}}}{m(u,s,\omega)}.
 $$
 Consequently, for each $t\in(\delta,T]$
 \begin{align*}
  \langle(X(\cdot),h_{\eps}^u),(X(\cdot),h_{\eps}^v)\rangle_t(\omega)&-\langle(X(\cdot),h_{\eps}^u),(X(\cdot),h_{\eps}^v)\rangle_{\delta}(\omega)\\
  &=\int_{\delta}^t\frac{\I_{\{\tau_{u,v}(\omega)\leq s\}}}{m(u,s,\omega)}ds\quad \mbox{for all}\ \ \eps\in(0,\eps_0].
 \end{align*}
 Hence, by~\eqref{f_chat_lim},
 $$
  \langle X(u,\cdot),X(v,\cdot)\rangle_t(\omega)-\langle X(u,\cdot),X(v,\cdot)\rangle_{\delta}(\omega)
  =\int_{\delta}^t\frac{\I_{\{\tau_{u,v}(\omega)\leq s\}}}{m(u,s,\omega)}ds.
 $$
 Making $\delta\to 0$ and using the continuity of $\langle X(u,\cdot),X(v,\cdot)\rangle_t(\omega)$, $t\in[0,T]$, we obtain
 $$
  \langle X(u,\cdot),X(v,\cdot)\rangle_t(\omega)=\int_0^t\frac{\I_{\{\tau_{u,v}(\omega)\leq s\}}}{m(u,s,\omega)}ds,\quad t\in[0,T].
 $$
 It finishes the proof of the theorem.
 \end{proof}

 \begin{rem}\label{remark_conditions_C_implies_mart_prop_in_L}
  Note that if a $C[0,T]$-valued process $Y(u,\cdot)$, $u\in(0,b)$ with trajectories from $D((0,b),C[0,T])$ satisfies $(C1)-(C4)$, then Proposition~\ref{proposition_estimation_of_m} and the similar calculation as in the proof of Theorem~3.1~\cite{Konarovskyi_LDP:2015} give that $Y(\cdot,t)$, $t\in[0,T]$, is an $L_2^{\uparrow}[0,1]$-valued continuous square integrable martingale with the quadratic variation $(M')$. Thus, the conditions $(C1)-(C4)$ are equivalent $(M)$ or $(M')$.
 \end{rem}

\section{Estimations of the expectation of mass and diffusion rate}\label{mass_estimations}

Throughout this and the next sections we will suppose that $\{X(u,t),\ u\in(0,1),\ t\in[0,T]\}$ belongs to $D((0,1),C[0,T])$, satisfies $(C1)-(C4)$ and $X(\cdot,0)=g\in L_{2+\eps}^{\uparrow}[0,1]$ for some $\eps>0$.

\subsection{Estimation of the expectation of diffusion rate}

\begin{prp}\label{prp_estimation_of_diff}
 Under the assumptions of Theorem~\ref{prp_LIL_2}, there exists $C'=C'_{u_0,\alpha,g}>0$ such that
 $$
 \E\frac{1}{m(u_0,t)}\leq \frac{C'}{t^{\frac{1}{2\alpha+1}}},\quad t\in(0,T].
 $$
\end{prp}

\begin{proof}
 To prove the proposition, we will use the estimation of $\p\{m(u,t)<r\}$ (see Lemma~\ref{lemma_estimation_of_p_m}). Assume that $t\in(0,T]$ is fixed and $g(u_0+u)-g(u_0)\leq Cu^{\alpha}$, $u\in[0,\delta]$, for some $C>0$ and $\delta>0$. For the case  $g(u_0)-g(u_0-u)\leq Cu^{\alpha}$ the proof is similar.  We estimate
 \begin{align*}
  \E\frac{1}{m(u_0,t)}&=\int_0^{+\infty}\p\left\{\frac{1}{m(u_0,t)}>r\right\}dr=\int_0^{+\infty}\p\left\{m(u_0,t)<\frac{1}{r}\right\}dr\\
  &\leq C_1+\frac{2}{\sqrt{2\pi}}\int_{\frac{1}{\delta}}^{+\infty}dr\int_0^{\frac{g(u_0+1/r)-g(u_0)}{\sqrt{rt}}}e^{-\frac{x^2}{2}}dx=I.
 \end{align*}
 Set
 $$
 x_t(r)=\frac{g(u_0+1/r)-g(u_0)}{\sqrt{rt}},\quad r\in[1/\delta,\infty),
 $$
 and note that $x_t$ strictly decreases to zero for each $t$. Consequently, there exists the inverse map
 $$
 r_t(x)=\max\{r:\ x_t(r)\geq x\},\quad x\in\left[0,c/\sqrt{t}\right],
 $$
 where $c=(g(u_0+\delta)-g(u_0))\sqrt{\delta}$.
 Hence, interchanging of integrations, we obtain
 \begin{align*}
 I=&C_1+\frac{2}{\sqrt{2\pi}}\int_0^{\frac{c}{\sqrt{t}}}e^{-\frac{x^2}{2}}\left(\int_{\frac{1}{\delta}}^{r_t(x)}dr\right)dx\\
 &=C_1+\frac{2}{\sqrt{2\pi}}\int_0^{\frac{c}{\sqrt{t}}}e^{-\frac{x^2}{2}}\left(r_t(x)-\frac{1}{\delta}\right)dx.
 \end{align*}

 Next, for fixed $x'\in\left[0,c/\sqrt{t}\right]$ we denote
 $$
 r'=r_t(x')=\max\left\{r:\ \frac{g(u_0+1/r)-g(u_0)}{\sqrt{r}}\geq x'\sqrt{t}\right\}.
 $$
 By assumption of the proposition, for all $r\geq\frac{1}{\delta}$
 \begin{equation}\label{f_estim_of_r}
 \frac{g(u_0+1/r)-g(u_0)}{\sqrt{r}}\leq\frac{C}{r^{\frac{2\alpha+1}{2}}}.
 \end{equation}
 Thus, using the inequality $\frac{g(u_0+1/{r'})-g(u_0)}{\sqrt{r'}}\geq x'\sqrt{t}$ and~\eqref{f_estim_of_r}, we have
 $$
 r'\leq\frac{C_2}{(x')^{\frac{2}{2\alpha+1}}t^{\frac{1}{2\alpha+1}}},
 $$
 where $C_2=C^{\frac{2}{2\alpha+1}}$. So,
 $$
 r_t(x)\leq\frac{C_2}{x^{\frac{2}{2\alpha+1}}t^{\frac{1}{2\alpha+1}}},\quad x\in\left(0,c/\sqrt{t}\right].
 $$
 Let us come back to the estimation of $\E\frac{1}{m(u_0,t)}$. So,
 \begin{align*}
  \E\frac{1}{m(u_0,t)}&\leq C_1+\frac{2}{\sqrt{2\pi}}\int_0^{\frac{c}{\sqrt{t}}}e^{-\frac{x^2}{2}}\left(\frac{C_2}{x^{\frac{2}{2\alpha+1}}t^{\frac{1}{2\alpha+1}}}-\frac{1}{\delta}\right)dx\\
  &\leq t^{-\frac{1}{2\alpha+1}}\left\{C_1t^{\frac{1}{2\alpha+1}}+\frac{2}{\sqrt{2\pi}}\int_0^{+\infty}e^{-\frac{x^2}{2}}\left(\frac{C_2}{x^{\frac{2}{2\alpha+1}}}-\frac{t^{\frac{1}{2\alpha+1}}}{\delta}\right)dx\right\}\\
  &\leq C_3t^{-\frac{1}{2\alpha+1}},\quad t\in(0,T],
 \end{align*}
 since the integral in the brackets $\{\cdot\}$ is finite for $\alpha>\frac{1}{2}$. The proposition is proved.
\end{proof}

\subsection{Rescaling property of $X$}

In this subsection we prove that conditions $(C1)-(C4)$ is invariant with respect to the transformation $(u,t)\to(\rho u,\rho^{\gamma}t)$. So, let $\rho>0$, $q\in\R$ and $\alpha>0$ be fixed. Set
\begin{equation}\label{f_X_rho}
X_{\rho}(u,t)=\frac{1}{\rho^{\alpha}}X(u\rho-q,t\rho^{\gamma}),\quad u\in(q/\rho,(q+1)/\rho),\ \ t\in[0,T/\rho^{\gamma}],
\end{equation}
where $\gamma=2\alpha+1$.
\begin{lem}\label{lem_rescaling_property}
 Let $a=\frac{q}{\rho}$, $b=\frac{q+1}{\rho}$. The process $X_{\rho}$ belongs to $D((a,b),C[0,T/\rho^{\gamma}])$ and satisfies $(C1)-(C4)$ with the function $\frac{1}{\rho^{\alpha}}g(\cdot\rho-q)$ instead of $g$ in condition~$(C2)$. Moreover,
 \begin{align}\label{f_m_rescaling}
 \begin{split}
 m_{\rho}(u,t)&=\leb\{w:\ \exists s\leq t\ X_{\rho}(u,s)=X_{\rho}(w,s)\}\\
 &=\frac{1}{\rho}m(u\rho-q,t\rho^{\gamma})\quad\mbox{for all}\ \  u\in(a,b),\ \ t\in[0,T/\rho^{\gamma}].
  \end{split}
 \end{align}
\end{lem}

\begin{proof}
 The proof of $(C1)-(C3)$ are trivial. We will only check $(C4)$. Let $u,v\in(q/\rho,(q+1)/\rho)$ and $t\in[0,T/\rho^{\gamma}]$. We first evaluate
 \begin{align*}
  m_{\rho}(u,t)&=\leb\left\{w:\ \exists s\leq t\ \ X_{\rho}(u,s)=X_{\rho}(w,s)\right\}\\
  &=\leb\left\{w:\ \exists s\leq t\ \ \frac{1}{\rho^{\alpha}}X(u\rho-q,s\rho^{\gamma})=\frac{1}{\rho^{\alpha}}X(w\rho-q,s\rho^{\gamma})\right\}\\
  &=\leb\left\{\frac{w+q}{\rho}:\ \exists s\leq t\rho^{\gamma}\ \ X(u\rho-q,s)=X(w,s)\right\}\\
  &=\frac{1}{\rho}\leb\left\{w:\ \exists s\leq t\rho^{\gamma}\ \ X(u\rho-q,s)=X(w,s)\right\}\\
  &=\frac{1}{\rho}m(u\rho-q,t\rho^{\gamma}).
 \end{align*}
Next, let $\tau_{u,v}^{(\rho)}=\inf\{t:\ X_{\rho}(u,t)=X_{\rho}(v,t)\}\wedge \frac{T}{\rho^{\gamma}}$. It is easily seen that
$$
\tau_{u,v}^{(\rho)}=\frac{\tau_{u\rho-q,v\rho-q}}{\rho^{\gamma}},
$$
where $\tau_{u',v'}=\inf\{t:\ X(u',t)=X(v',t)\}\wedge T$, \ $u',v'\in(0,1)$.
Then
 \begin{align*}
 \langle X_{\rho}(u,\cdot),&X_{\rho}(v,\cdot)\rangle_t=\left\langle\frac{1}{\rho^{\alpha}}X(u\rho-q,\cdot\rho^{\gamma}),\frac{1}{\rho^{\alpha}}X(v\rho-q,\cdot\rho^{\gamma})\right\rangle_t=\\
 &=\frac{1}{\rho^{2\alpha}}\left\langle X(u\rho-q,\cdot),X(v\rho-q,\cdot)\right\rangle_{t\rho^{\gamma}}=\\
 &=\frac{1}{\rho^{2\alpha}}\int_0^{t\rho^{\gamma}}\frac{\I_{\{\tau_{u\rho-q,v\rho-q}\leq s\}}}{m(u\rho-q,s)}ds\\
 &=\rho^{\gamma-2\alpha}\int_0^t\frac{\I_{\{\tau_{u\rho-q,v\rho-q}\leq s\rho^{\gamma}\}}}{m(v\rho-q,s\rho^{\gamma})}ds
 =\int_0^t\frac{\I_{\{\tau_{u,v}^{(\rho)}\leq s\}}ds}{m_{\rho}(u,s)},
 \end{align*}
if $\gamma-2\alpha=1$.
It finishes the proof of the lemma.
\end{proof}

\subsection{Estimation of the expectation of mass}

Let $u_0\in(0,1)$ be fixed. In this subsection we estimate the expectation $\E m(u_0,t)$ in the case where the function $g(u_0+\cdot)-g(u_0)$ is locally (at zero) similar to $|u|^\alpha$, $u\leq\delta$. To get the estimation we use the rescaling property of $X$. The following statement holds.

\begin{lem}\label{lem_boundedness_of_m}
 Let $X_{\rho}$ be defined by~\eqref{f_X_rho} with $q=-u_0$ and there exists $C>0$ such that $\E m_{\rho}(0,T)\leq C$ for all $\rho\in(0,1]$. Then
 $$
 \E m(u_0,t)\leq C t^{\frac{1}{2\alpha+1}},\quad t\in(0,T].
 $$
\end{lem}

\begin{proof}
 By~\eqref{f_m_rescaling},
 $$
 \E\frac{1}{\rho}m(u_0,T\rho^{\gamma})=\E m_{\rho}(0,T)\leq C,
 $$
 where $\gamma=2\alpha+1$. So,
 $$
  \E m(u_0,T\rho^{\gamma})\leq C\rho,\quad \rho\in(0,1].
 $$
 It implies the needed estimation if $t=T\rho^{\gamma}\in(0,T]$.
\end{proof}

\begin{prp}\label{prp_estimation_Em}
 Under the assumptions of Theorem~\ref{prp_LIL_1}, there exists $C'=C'_{u_0,\alpha,g}>0$ such that
 \begin{equation}\label{f_est_of_mass}
 \E m(u_0,t)\leq C't^{\frac{1}{2\alpha+1}},\quad t\in[0,T].
 \end{equation}
\end{prp}

\begin{proof}
According to Lemma~\ref{lem_boundedness_of_m}, it is enough to show the boundedness of $\E m_{\rho}(0,T)$. Let $\gamma=2\alpha+1$, $a=-u_0$, $b=(1-u_0)$. Here we will consider $X_{\rho}(u,t)$, $u\in(a/\rho,b/\rho)$, $t\in[0,T/\rho^{\gamma}]$, as a process from $D((a/\rho,b/\rho),C[0,T/\rho^{\gamma}])$ which satisfies $(C1)-(C4)$. Note that according to $(C2)$,
\begin{equation}\label{f_X_rho_0}
 X_{\rho}(u,0)=g_{\rho}(u)=\frac{1}{\rho^{\alpha}}g(u\rho+u_0),\quad u\in(a/\rho,b/\rho).
\end{equation}

Without loss of generality, we may suppose that $g_{\rho}(0)=\frac{1}{\rho^{\alpha}}g(u_0)=0$. If $g(u_0)\neq 0$, then we can consider the new process $X_{\rho}(u,t)-g_{\rho}(u_0)$, $u\in(a/\rho,b/\rho)$, $t\in[0,T/\rho^{\gamma}]$, instead of $X_{\rho}$.

By~\eqref{f_coalescing}, we have
\begin{align*}
 \E m_{\rho}(0,T)&=\E\int_{\frac{a}{\rho}}^{\frac{b}{\rho}}\I_{\{X_{\rho}(0,T)=X_{\rho}(u,T)\}}du
 =\int_{\frac{a}{\rho}}^{\frac{b}{\rho}}\p\left\{X_{\rho}(0,T)=X_{\rho}(u,T)\right\}du\\
 &=\int_{\frac{a}{\rho}}^0\p\left\{X_{\rho}(0,T)=X_{\rho}(u,T)\right\}du+\int_0^{\frac{b}{\rho}}\p\left\{X_{\rho}(0,T)=X_{\rho}(u,T)\right\}du.
\end{align*}
Next we estimate only the integral $\int_0^{\frac{b}{\rho}}$. The integral $\int_{\frac{a}{\rho}}^0$ can be estimated by the same way.

Set $M_{\rho}(u,t)=X_{\rho}(u,t)-X_{\rho}(0,t)$ for all $u\in(0,b/\rho)$ and $t\in[0,T/\rho^{\gamma}]$. By $(C1)$, $M_{\rho}(u,\cdot)$ is a continuous square integrable martingale starting from $M_{\rho}(u,0)=g_{\rho}(u)-g_{\rho}(0)=\frac{1}{\rho^{\alpha}}g(u\rho+u_0)$.

Using the Paley-Zygmund inequality, we can estimate
\begin{align*}
 \p\{X_{\rho}(0,T)&=X_{\rho}(u,T)\}=1-\p\{M_{\rho}(u,T)>0\}\\
 &\leq 1-\frac{\left(\E M_{\rho}(u,T)\right)^2}{\E M_{\rho}^2(u,T)}=1-\frac{g_{\rho}^2(u)}{\E M_{\rho}^2(u,T)}\\
 &=1-\frac{g_{\rho}^2(u)}{\Var M_{\rho}(u,T)+g_{\rho}^2(u)}=\frac{\Var M_{\rho}(u,T)}{\Var M_{\rho}(u,T)+g_{\rho}^2(u)}.
\end{align*}

Next, we will estimate $\Var M_{\rho}(u,T)$. By $(C4)$,
\begin{align*}
 \Var M_{\rho}(u,T)&=\E\left(M_{\rho}(u,T)-g_{\rho}(u)\right)^2\\
 &=\E\int_0^T\left(\frac{1}{m_{\rho}(u,s)}+\frac{1}{m_{\rho}(0,s)}-\frac{\I_{\{\tau_{0,u}^{(\rho)}\leq s\}}}{m_{\rho}(u,s)}\right)ds\\
 &\leq \int_0^T\E\frac{1}{m_{\rho}(u,s)}ds+\int_0^T\E\frac{1}{m_{\rho}(0,s)}ds=\xi_{\rho}(u)+\xi_\rho(0),
\end{align*}
where $\xi_{\rho}(u)=\int_0^T\E\frac{1}{m_{\rho}(u,s)}ds$.

It should be noted that $\xi_{\rho}(0)$, $\rho\in(0,1]$, is bounded. Indeed, inserting $\varsigma=u-u_0$ in $(i)$, we can see that
$$
g(u)-g(u_0)\leq C(u-u_0)^{\alpha}\quad \mbox{if}\ \ u-u_0\leq\delta.
$$
Thus, by Proposition~\ref{prp_estimation_of_diff}, we have the estimation
$$
 \E\frac{1}{m(u_0,t)}\leq \frac{C'}{t^{\frac{1}{2\alpha+1}}},\quad t\in(0,T].
$$
Hence, for all $\rho\in(0,1]$
\begin{align*}
 \xi_{\rho}(0)&=\int_0^T\E\frac{1}{m_{\rho}(0,s)}ds=\rho\int_0^T\E\frac{1}{m(u_0,s\rho^{\gamma})}ds\\
 &\leq C'\rho\int_0^T\frac{1}{(s\rho^{\gamma})^{\frac{1}{2\alpha+1}}}ds=C'\int_0^T\frac{1}{s^{\frac{1}{2\alpha+1}}}ds\leq c,
\end{align*}
where $c$ is a constant.

Consequently,
\begin{align*}
 \p\{X_{\rho}(0,T)&=X_{\rho}(u,T)\}\leq\frac{\xi_{\rho}(u)+c}{\xi_{\rho}(u)+c+g_{\rho}^2(u)}.
\end{align*}
Hence,
\begin{align*}
 \int_0^{\frac{b}{\rho}}\p\{X_{\rho}(0,T)&=X_{\rho}(u,T)\}du\leq\int_0^{\frac{b}{\rho}}\frac{\xi_{\rho}(u)+c}{\xi_{\rho}(u)+c+g_{\rho}^2(u)}du\\
 &\leq b+\int_b^{\frac{b}{\rho}}\frac{\xi_{\rho}(u)+c}{c+g_{\rho}^2(u)}du.
\end{align*}
First we estimate $\int_b^{\frac{b}{\rho}}\frac{c}{c+g_{\rho}^2(u)}du$, using $(ii)$. Let $b\rho<\delta$. Then
\begin{align*}
 \int_b^{\frac{b}{\rho}}\frac{c}{c+g_{\rho}^2(u)}du&=\int_b^{\frac{b}{\rho}}\frac{c}{c+\frac{1}{\rho^{2\alpha}}g^2(u\rho+u_0)}du=\frac{1}{\rho}\int_{b\rho}^b\frac{c}{c+\frac{1}{\rho^{2\alpha}}g^2(u+u_0)}du\\
 &=\rho^{2\alpha-1}\int_{b\rho}^b\frac{c}{c\rho^{2\alpha}+g^2(u+u_0)}du\leq \rho^{2\alpha-1}\int_{b\rho}^b\frac{c}{g^2(u+u_0)}du\\
 &\leq\rho^{2\alpha-1}\int_{\delta}^b\frac{c}{g^2(\delta+u_0)}du+\rho^{2\alpha-1}\int_{b\rho}^{\delta}\frac{c}{g^2(u+u_0)}du\\
 &\leq C_1\rho^{2\alpha-1}+\rho^{2\alpha-1}\int_{b\rho}^{\delta}\frac{c}{u^{2\alpha}}du\leq C_2,
\end{align*}
since $\alpha>\frac{1}{2}$.

Next we estimate $\int_b^{\frac{b}{\rho}}\frac{\xi_{\rho}(u)}{c+g_{\rho}^2(u)}du$. Note that if $\xi_{\rho}(u)$ was bounded with respect to $\rho$ and $u$ (e.g. it is true if $g(u)=u$, $u\in(0,1)$), then the integral would be bounded with respect to $\rho$. It would prove the proposition. In general, we should not expect that $\xi_{\cdot}(\cdot)$ is bounded, since it depends on local properties of $g_{\rho}$ at each point. So, in order to prove the boundedness of the integral, we will use Lemma~\ref{lemma_estimation_of_p_m}.

Since $\frac{1}{m_{\rho}(u,s)}\in[\rho,\infty)$ for all $s\in(0,T]$ and $u\in (b,b/\rho)$, we can estimate
\begin{align*}
 \E\frac{1}{m_{\rho}(u,s)}&=\int_{\rho}^{+\infty}\p\left\{m_{\rho}(u,s)<\frac{1}{r}\right\}dr\\
 &\leq\int_{\rho}^{\theta}dr+\int_{\theta}^{+\infty}\p\left\{m_{\rho}(u,s)<\frac{1}{r}\right\}dr\\
 &\leq\theta+\frac{2}{\sqrt{2\pi s}}\int_{\theta}^{+\infty}\frac{g_{\rho}(u)-g_{\rho}(u-1/r)}{\sqrt{r}}dr,
\end{align*}
where $\theta>\frac{1}{b}\vee 1$ is fixed.

Thus,
\begin{align*}
\int_b^{\frac{b}{\rho}}\frac{\xi_{\rho}(u)}{c+g_{\rho}^2(u)}du&\leq \int_b^{\frac{b}{\rho}}\frac{1}{c+g_{\rho}^2(u)}\left(\int_0^T\E\frac{1}{m_{\rho}(u,s)}ds\right)du\\
&\leq \int_b^{\frac{b}{\rho}}\frac{\theta T}{c+g_{\rho}^2(u)}du+\frac{4\sqrt{T}}{\sqrt{2\pi}}\int_b^{\frac{b}{\rho}}\left(\int_{\theta}^{+\infty}\frac{g_{\rho}(u)-g_{\rho}(u-1/r)}{\sqrt{r}g_{\rho}^2(u)}dr\right)du.
\end{align*}
Note that $\int_b^{\frac{b}{\rho}}\frac{\theta T}{c+g_{\rho}^2(u)}du$ is bounded with respect to $\rho$. Let us estimate
\begin{align*}
 \int_b^{\frac{b}{\rho}}&\left(\int_{\theta}^{+\infty}\frac{g_{\rho}(u)-g_{\rho}(u-1/r)}{\sqrt{r}g_{\rho}^2(u)}dr\right)du\\
 &=\int_b^{\frac{b}{\rho}}\left(\int_{\theta}^{+\infty}\frac{\frac{1}{\rho^{\alpha}}g(u\rho+u_0)-\frac{1}{\rho^{\alpha}}g((u-1/r)\rho+u_0)}{\sqrt{r}\frac{1}{\rho^{2\alpha}}g^2(u\rho+u_0)}dr\right)du\\
 &=\rho^{\alpha-\frac{1}{2}}\int_{b\rho}^b\left(\int_{\frac{\theta}{\rho}}^{+\infty}\frac{g(u+u_0)-g(u+u_0-1/r)}{\sqrt{r}g^2(u+u_0)}dr\right)du=I.
\end{align*}
Let $b\rho<\delta$. Then by $(i)$ and $(ii)$,
\begin{align*}
 I&\leq\frac{\rho^{\alpha-\frac{1}{2}}}{g^2(\delta+u_0)}\int_{\delta}^b\left(\int_{\frac{\theta}{\rho}}^{+\infty}\frac{g(u+u_0)-g(u+u_0-1/r)}{\sqrt{r}}dr\right)du\\
 &+\rho^{\alpha-\frac{1}{2}}\int_{b\rho}^{\delta}\left(\int_{\frac{\theta}{\rho}}^{+\infty}\frac{g(u+u_0)-g(u+u_0-1/r)}{\sqrt{r}g^2(u+u_0)}dr\right)du\\
 &\leq C_3\frac{\rho^{\alpha-\frac{1}{2}}}{g^2(\delta+u_0)}(1+\|g\|_{L_{2+\eps}})+C^2\rho^{\alpha-\frac{1}{2}}\int_{b\rho}^{\delta}\left(\int_{\frac{\theta}{\rho}}^{+\infty}\frac{u^{(\alpha-1)\vee 0}}{r^{\frac{1}{2}+\alpha\wedge 1}u^{2\alpha}}dr\right)du\\
 &\leq C_4\rho^{\alpha-\frac{1}{2}}+C^2\rho^{\alpha-\frac{1}{2}}\int_{b\rho}^{\delta}\left(\int_{\frac{\theta}{\rho}}^{+\infty}\frac{u^{(\alpha-1)\vee 0}}{r^{\frac{1}{2}+\alpha\wedge 1}u^{2\alpha}}dr\right)du.
\end{align*}
Here the integral $\int_{\delta}^b\left(\int_{\frac{\theta}{\rho}}^{+\infty}\frac{g(u+u_0)-g(u+u_0-1/r)}{\sqrt{r}}dr\right)du$ is estimated similarly as in the proof of Proposition~\ref{proposition_estimation_of_m}. Since $\alpha>\frac{1}{2}$, it is easy to see that the right hand side of the latter inequality is bounded by a constant that is independent of $\rho$. It finishes the proof of the proposition.
\end{proof}

\section{Asymptotic behavior (proofs of theorems~\ref{prp_LIL_1} and~\ref{prp_LIL_2})}\label{asymptotic_behavior}

\begin{proof}[Proof of Theorem~\ref{prp_LIL_1}]
 To prove the theorem, we will only use inequality~\eqref{f_est_of_mass} and the fact that $X(u_0,\cdot)$ is a continuous square integrable martingale with the quadratic variation $\langle X(u_0,\cdot)\rangle_t=\int_0^t\frac{ds}{m(u_0,s)}$, $t\in[0,T]$.

 Let $\theta>0$, $\lambda<1$ and $t_n=\lambda^n$, $n\in\N$. Set
 $$
 A_n=\{m(u_0,t)>\theta\varphi(t)\quad\mbox{for some}\ \ t\in(t_{n+1},t_n]\},
 $$
 where $\varphi(t)=t^{\frac{1}{2\alpha+1}}\left(\ln\frac{1}{t}\right)^{1+\epsilon}$, $t\in(0,T]$. Using the monotonicity of $m(u_0,\cdot)$ and $\varphi$, Chebyshev's inequality and Proposition~\ref{prp_estimation_Em}, we can estimate
 \begin{align*}
  \p\{A_n\}&\leq\p\{m(u_0,t_n)>\theta\varphi(t_{n+1})\}\leq\frac{1}{\theta\varphi(t_{n+1})}\E m(u_0,t_n)\\
  &\leq\frac{Ct_n^{\frac{1}{2\alpha+1}}}{\theta t_{n+1}^{\frac{1}{2\alpha+1}}\left(\ln\frac{1}{t_{n+1}}\right)^{1+\epsilon}}=\frac{C}{\theta \lambda^{\frac{1}{2\alpha+1}}\left(\ln\frac{1}{\lambda}\right)^{1+\epsilon}(n+1)^{1+\epsilon}}.
 \end{align*}
 Hence, $\sum_{n=1}^{\infty}\p\{A_n\}$ converges and consequently, by the Borel-Cantelli lemma,
 $$
 \p\left\{\varlimsup_{n\to\infty}A_n\right\}=0.
 $$
 It implies~\eqref{f_LIL_m_1}.

 Next we will prove~\eqref{f_LIL_Y_1}. By Theorem~2.7.2'~\cite{Watanabe:1981:en}, there exists a Wiener process $B(t)$, $t\geq 0$, (maybe on an extended probability space) such that
 \begin{equation}\label{f_representation_of_X}
 X(u_0,t)=g(u_0)+B(\langle X(u_0,\cdot)\rangle_t),\quad t\in[0,T].
 \end{equation}
 So, to get~\eqref{f_LIL_Y_1}, we will use the law of the iterated logarithm for the Wiener process (see e.g. Theorem~13.18~\cite{Kallenberg:2002}) and the latter relation.

 We first estimate the quadratic variation of $X(u_0,\cdot)$. Set
 $$
 \Omega'=\left\{\exists t_0>0:\ \forall t\in(0,t_0]\ \ m(u_0,t)\leq t^{\beta}\left(\ln\frac{1}{t}\right)^{1+2\epsilon}\right\},
 $$
 where $\beta=\frac{1}{2\alpha+1}$. By~\eqref{f_LIL_m_1}, $\p\{\Omega'\}=1$. So, let $\omega\in\Omega'$ and $t\in(0,t_0(\omega)]$. Then
 $$
 \langle X(u_0,\cdot)\rangle_t(\omega)=\int_0^t\frac{ds}{m(u_0,s,\omega)}\geq\int_0^t\frac{ds}{s^{\beta}\left(\ln\frac{1}{s}\right)^{1+2\epsilon}}.
 $$
  By L'Hopital's rule, we obtain
 $$
 \lim_{t\to 0}\frac{\int_0^t\frac{ds}{s^{\beta}\left(\ln\frac{1}{s}\right)^{1+2\epsilon}}}{t^{1-\beta}\left(\ln\frac{1}{t}\right)^{-1-2\epsilon}}=\frac{1}{1-\beta}.
 $$
 Hence, there exists $t_1(\omega)\in(0,t_0(\omega)]$ such that
 $$
 \langle X(u_0,\cdot)\rangle_t(\omega)\geq\frac{1}{2(1-\beta)}t^{1-\beta}\left(\ln\frac{1}{t}\right)^{-1-2\epsilon}\quad\mbox{for all}\ \  t\in(0,t_1(\omega)].
 $$
 Thus, using the law of the iterated logarithm for the Wiener process and~\eqref{f_representation_of_X}, we have almost surely
 \begin{align*}
  1&=\varlimsup_{t\to 0}\frac{|X(u_0,t)-g(u_0)|}{\sqrt{2\langle X(u_0,\cdot)\rangle_t\ln\ln\frac{1}{\langle X(u_0,\cdot)\rangle_t}}}\leq \varlimsup_{t\to 0}\frac{|X(u_0,t)-g(u_0)|}{\sqrt{2\langle X(u_0,\cdot)\rangle_t}}\\
  &\leq \varlimsup_{t\to 0}\frac{\sqrt{(1-\beta)}|X(u_0,t)-g(u_0)|}{t^{\frac{1-\beta}{2}}\left(\ln\frac{1}{t}\right)^{-\frac{1}{2}-\epsilon}}.
 \end{align*}
 Since $\epsilon>0$ is arbitrary, we obtain~\eqref{f_LIL_Y_1}. The theorem is proved.
\end{proof}

\begin{proof}[Proof of Theorem~\ref{prp_LIL_1}]
 The statement follows from Proposition~\ref{prp_estimation_of_diff} and the same argument as in the proof of Theorem~\ref{prp_LIL_1}.
\end{proof}

\begin{rem}\label{rem_asymprotic_bahavior_of_Z}
 If $X$ is the process constructed in~\cite{Konarovskyi:2014:arx}, i.e. $g(u)=u$, $u\in[0,1]$, then for all $u\in(0,1)$ and $\epsilon>0$
 \begin{align*}
  \p\left\{\lim_{t\to 0}\frac{m(u,t)}{\sqrt[3]{t}\left(\ln\frac{1}{t}\right)^{1+\epsilon}}=0\right\}
  &=\p\left\{\lim_{t\to 0}\frac{m(u,t)}{\sqrt[3]{t}\left(\ln\frac{1}{t}\right)^{-1-\epsilon}}=+\infty\right\}=1,\\
  \p\left\{\lim_{t\to 0}\frac{|X(u,t)-u|}{\sqrt[3]{t}\left(\ln\frac{1}{t}\right)^{\frac{1}{2}+\epsilon}}=0\right\}
  &=\p\left\{\varlimsup_{t\to 0}\frac{|X(u,t)-u|}{\sqrt[3]{t}\left(\ln\frac{1}{t}\right)^{-\frac{1}{2}-\epsilon}}=+\infty\right\}=1.
 \end{align*}
\end{rem}

\appendix
\section{The space of non-decreasing functions}

Let $a,b\in\R$ and $a<b$. For $c\in(a,b)$ and $\eps>0$ we set
$$
h_{c,\eps}=\frac{1}{\eps\wedge(b-c)}\I_{[c,c+\eps\wedge (b-c)]}.
$$

\begin{prp}\label{prp_L2_Uparrow_crit}
 A function $g\in L_2[a,b]$ belongs to $L_2^{\uparrow}[a,b]$ if and only if
 \begin{enumerate}
  \item[(A)] for all $c_1,c_2\in(a,b)$ and $\eps_1,\eps_2>0$ satisfying $c_1+\eps_1\leq c_2$
 $$
 (g,h_{c_1,\eps_1})\leq(g,h_{c_2,\eps_2}).
 $$
 \end{enumerate}
 Moreover, the modification $\widetilde{g}$ of $g$ from $D^{\uparrow}$ is given as follows
 $$
 \widetilde{g}(u)=\lim_{\eps\to 0+}(g,h_{u,\eps}),\quad u\in(a,b),
 $$
 and
 $$
\widetilde{g}(a)=\lim_{u\to a+}\widetilde{g}(u),\quad\widetilde{g}(b)=\lim_{u\to b-}\widetilde{g}(u).
$$
\end{prp}

\begin{cor}\label{cor_clasability}
  The set $L_2^{\uparrow}[a,b]$ is closed in $L_2[a,b]$.
\end{cor}

To prove the proposition we will prove several auxiliary lemmas.

\begin{lem}\label{lem_boundedness_of_g}
 Let $g$ satisfy $(A)$. Then for each $\delta>0$ there exists $C>0$ such that $|(g,h_{c,\eps})|\leq C$ for all $c\in(a+\delta,b-2\delta)$ and $\eps<\delta$.
\end{lem}

\begin{proof}
  By $(A)$, we have $(g,h_{a,\delta})\leq(g,h_{c,\eps})\leq (g,h_{b-\delta,\delta})$.
\end{proof}

\begin{lem}
 If $g$ satisfies $(A)$, then for each $c\in(a,b)$ and $0<\eps'<\eps$
 $$
 (g,h_{c,\eps'})\leq(g,h_{c,\eps}).
 $$
\end{lem}

\begin{proof}
  The inequality follows from the following simple algebraic transformations
  \begin{align*}
    (g,h_{c,\eps})&=\frac{\eps'}{\eps}(g,h_{c,\eps'})+\frac{\eps-\eps'}{\eps}(g,h_{c+\eps',\eps-\eps'})\\
    &\geq \frac{\eps'}{\eps}(g,h_{c,\eps'})+\frac{\eps-\eps'}{\eps}(g,h_{c,\eps'})=(g,h_{c,\eps'}).
  \end{align*}
\end{proof}

\begin{lem}\label{lem_continuity}
  Let $u,u_n\in(a,b)$, $n\geq 1$, and $u_n\to u$. Then for each $\eps>0$, $(g,h_{u_n,\eps})\to(g,h_{u,\eps})$.
\end{lem}

\begin{proof}
  The statement follows from the Cauchy-Schwarz inequality. Indeed,
  $$
  |(g,h_{u_n,\eps})-(g,h_{u,\eps})|\leq\|h_{u_n,\eps}-h_{u,\eps}\|_{L_2}\|g\|_{L_2}\to 0\quad\mbox{as}\ \ n\to\infty.
  $$
\end{proof}

\begin{proof}[Proof of Proposition~\ref{prp_L2_Uparrow_crit}]
We note that if $g\in D^{\uparrow}$, then it is easily seen that $(A)$ holds. So, we need to show that $(A)$ implies that $g$ has a modification from $D^{\uparrow}$.
From the previous lemmas the sequence $\{(g,h_{u,\eps})\}_{\eps>0}$ is bounded and decreasing for all $u\in(a,b)$. Consequently, there exists a limit
$$
\widetilde{g}(u)=\lim_{\eps\to 0+}(g,h_{u,\eps})\quad\mbox{for all}\ \ u\in(a,b).
$$
From $(A)$ it follows that $\widetilde{g}(u)$, $u\in(a,b)$, is increasing. So, we can set
$$
\widetilde{g}(a)=\lim_{u\to a+}\widetilde{g}(u)\quad\mbox{and}\quad\widetilde{g}(b)=\lim_{u\to b-}\widetilde{g}(u).
$$
Let us show that $\widetilde{g}$ belongs to $D^{\uparrow}$. Take $r>0$, $u\in(a,b)$ and a sequence $\{u_n\}_{n\geq 1}\subset (a,b)$ such that $u_n\downarrow u$. Then there exists $\eps\in(0,b-u)$ such that
$$
(g,h_{u,\eps})-\widetilde{g}(u)<\frac{r}{2}.
$$
Since $u_n\downarrow u$, there exists $N$ such that
$$
(g,h_{u_n,\eps})-(g,h_{u,\eps})<\frac{r}{2}
$$
for all $n\geq N$, by Lemma~\ref{lem_continuity}. Thus,
\begin{align*}
  \widetilde{g}(u_n)\leq (g,h_{u_n,\eps})<(g,h_{u,\eps})+\frac{r}{2}<\widetilde{g}(u)+r.
\end{align*}
Since $r$ is arbitrary, $\widetilde{g}(u_n)\to\widetilde{g}(u)$ and we hence obtain that $\widetilde{g}$ belongs to $D^{\uparrow}$.

To finish the proof, we have to show that $g=\widetilde{g}$ a.e. Let $\delta>0$ be fixed. First we note that by Lemma~\ref{lem_boundedness_of_g}, $\widetilde{g}$ is bounded on $[a+\delta,b-\delta]$. Take $h\in C[a+\delta,b-\delta]$ and denote the restrictions of $g$ and $\widetilde{g}$ on $[a+\delta,b-\delta]$ by $g_{\delta}$ and $\widetilde{g}_{\delta}$, respectively. Then by the monotone convergence theorem, we obtain
\begin{align*}
  (\widetilde{g}_{\delta},h)&=\lim_{\eps\to 0}\int_{a+\delta}^{b-\delta}(g,h_{u,\eps})h(u)du\\
  &=\lim_{\eps\to 0}\int_{a+\delta}^{b-\delta}\left(\int_a^b\frac{1}{\eps\wedge(b-u)}g(v)\I_{[u,u+\eps\wedge (b-u)]}(v)h(u)dv\right)du\\
  &=\lim_{\eps\to 0}\int_a^bg(v)\left(\int_{a+\delta}^{b-\delta}\frac{1}{\eps\wedge(b-u)}\I_{[u,u+\eps\wedge (b-u)]}(v)h(u)du\right)dv.
\end{align*}
Using the continuity of $h$, we have that
$$
\int_{a+\delta}^{b-\delta}\frac{1}{\eps\wedge(b-u)}\I_{[u,u+\eps\wedge (b-u)]}(v)h(u)du\to h(v)\quad\mbox{as}\ \ \eps\to 0
$$
for all $v\in(a+\delta,b-\delta)$, by the mean value theorem. Using the dominated convergence theorem, we get
$$
\lim_{\eps\to 0}\int_a^bg(v)\left(\int_{a+\delta}^{b-\delta}\frac{1}{\eps\wedge(b-u)}\I_{[u,u+\eps\wedge (b-u)]}(v)h(u)du\right)dv=(g_{\delta},h).
$$
Since $C[a+\delta,b-\delta]$ is dense in $L_2[a+\delta,b-\delta]$, $g_{\delta}=\widetilde{g}_{\delta}$ a.e. Making $\delta\to 0$, we obtain that $g=\widetilde{g}$ a.e. The proposition is proved.
\end{proof}

\begin{rem}\label{rem_unique_modification}
 It should be noted that the function $\widetilde{g}$, constructed in the proof of Proposition~\ref{prp_L2_Uparrow_crit}, is the unique modification of $g$ that belongs to $D^{\uparrow}$.
\end{rem}

Next we give a characterization of the inequality $\Pi_g\leq\Pi_f$ via the functionals $(\cdot,h_{u,\eps})$. Let us recall that $\Pi_g\leq\Pi_f$ if for each $\pi\in\Pi_f$ there exists $\pi'\in\Pi_g$ such that $\Int\pi\subseteq\pi'$. Let $D$ be dense in $[a,b]$.

\begin{lem}\label{lemma_conditions_for_inequality_for_Pi}
 Let $f,g\in L_2^{\uparrow}$. Then $\Pi_g\leq\Pi_f$ if and only if for all $u_1,u_2\in D$ and $\eps_1,\eps_2\in(0,b-a)\cap\Q$ such that $u_1+\eps_1\leq u_2$, the equality $(f,h_{u_1,\eps_1})=(f,h_{u_2,\eps_2})$ implies $(g,h_{u_1,\eps_1})=(g,h_{u_2,\eps_2})$.
\end{lem}

\begin{proof}
  Let $u_i$, $\eps_i$, $i=1,2$, be as in the assumption of the statement, $(f,h_{u_1,\eps_1})=(f,h_{u_2,\eps_2})$ and $\Pi_g\leq\Pi_f$. Since $f$ and $g$ belong to $L_2^{\uparrow}$, we may suppose that $f,g\in D^{\uparrow}$. Then by monotonicity of $f$, $f(u_1)=f(v)$ for all $v\in[u_1,u_2+\eps_2)$. Hence, there exists $\pi\in\Pi_f$ such that $[u_1,u_2+\eps_2)\subseteq\pi$. Next using definition of the partial order between $\Pi_f$ and $\Pi_g$, there exists $\pi'\in\Pi_g$ such that $(u_1,u_2+\eps_2)\subseteq\Int\pi\subseteq\pi'$. So, $g(u_1)=g(u_1+)=g(v)$ for all $v\in(u_1,u_2+\eps_2)$. Thus, it implies $(g,h_{u_1,\eps_1})=(g,h_{u_2,\eps_2})$.

  Conversely, let for all $u_1,u_2\in D$ and $\eps_1,\eps_2\in(0,b-a)\cap\Q$ such that $u_1+\eps_1\leq u_2$, the equality $(f,h_{u_1,\eps_1})=(f,h_{u_2,\eps_2})$ implies $(g,h_{u_1,\eps_1})=(g,h_{u_2,\eps_2})$. Taking $\pi\in\Pi_f$, we then have that $f(u)=f(v)$ for all $u,v\in\pi$. It implies that for all $u_i$, $\eps_i$, $i=1,2$, satisfying assumption of the statement and $(u_1,u_1+\eps_1)\cup(u_2,u_2+\eps_2)\subseteq\pi$, $(f,h_{u_1,\eps_1})=(f,h_{u_2,\eps_2})$. Consequently, $(g,h_{u_1,\eps_1})=(g,h_{u_2,\eps_2})$. Since $D$ is dense in $[a,b]$ and $g$ is a monotone function, $g(u)=g(v)$ for all $u,v\in\Int\pi$. Hence, there exists $\pi'\in\Pi_g$ such that $\Int\pi\subseteq\pi'$. It finishes the proof of the lemma.
\end{proof}

\section{Proof of propositions~\ref{prp_coalescing} and~\ref{prp_modif_from_D_for_mart}}\label{appendix_modif_from_D_for_mart}

In this section we use the notation from the previous one.

\begin{proof}[Proof of Proposition~\ref{prp_coalescing}]
For every $u\in(a,b)\cap\Q$ and $\eps\in(0,b-a)\cap\Q$ we set
$$
M_{u,\eps}(t)=(X(t),h_{u,\eps}),\quad t\in[0,T],
$$
and
$$
F=\{(u_1,u_2,\eps_1,\eps_2)\in((a,b)\cap\Q)^2\times((0,b-a)\cap\Q)^2:\ u_1+\eps_1\leq u_2\}.
$$
Note that $F$ is countable. Let
$$
\Omega'=\bigcap_{(u_1,u_2,\eps_1,\eps_2)\in F}\{\forall s\leq t\ \mbox{if } M_{u_1,\eps_1}(s)=M_{u_2,\eps_2}(s),\ \mbox{then } M_{u_1,\eps_1}(t)=M_{u_2,\eps_2}(t)\}.
$$

Since $M_{u_i,\eps_i}$, $i=1,2$, are continuous martingales and $M_{u_1,\eps_1}(t)\leq M_{u_2,\eps_2}(t)$, $t\in[0,T]$, Proposition~2.3.4~\cite{Revuz:1999} implies
$$
\p\{\forall s\leq t\ \mbox{if } M_{u_1,\eps_1}(s)=M_{u_2,\eps_2}(s),\ \mbox{then } M_{u_1,\eps_1}(t)=M_{u_2,\eps_2}(t)\}=1,\quad (u_1,u_2,\eps_1,\eps_2)\in F.
$$
By the countability of $F$, $\p\{\Omega'\}=1$. Next, Lemma~\ref{lemma_conditions_for_inequality_for_Pi} easily yields that for all $\omega\in\Omega'$ and $s\leq t$, $\Pi_{X(t,\omega)}\leq\Pi_{X(s,\omega)}$. The proposition is proved.
\end{proof}

\begin{proof}[Proof of Proposition~\ref{prp_modif_from_D_for_mart}]

To prove the proposition, we are going to construct the process~$\widetilde{X}$. Let $\{r_n,\ n\in\N\}\subset(0,T]$, $r_n\downarrow 0$ and
$$
\Omega'=\left\{\forall n\in\N\ \ \ X(r_n)\in\St\right\}
  \cap\left\{\mbox{for all } s\leq t,\ \Pi_{X(t)}\leq\Pi_{X(s)} \right\}\cap\{X\ \mbox{is continuous}\}.
$$
By Proposition~\ref{prp_coalescing}, $\p\{\Omega'\}=1$. Note that
$$
X_\eps(u,t,\omega)=(X(t,\omega),h_{u,\eps}),\quad t\in[0,T],
$$
is continuous for all $u\in(a,b)$ and $\omega\in\Omega'$. Moreover, since $X(r_n,\omega)\in\St$ and $\Pi_{X(t,\omega)}\leq\Pi_{X(r_n,\omega)}$ for all $t\in[r_n,T]$, there exists $\eps_0=\eps_0(u,n,\omega)>0$ such that
$$
X_{\eps'}(u,t,\omega)=X_{\eps''}(u,t,\omega)
$$
for all $0<\eps',\eps''\leq\eps_0$, $t\in[r_n,T]$, $u\in(a,b)$ and $\omega\in\Omega'$. We set
$$
\widetilde{X}(u,t,\omega)=\begin{cases}
                            g(u),& \omega\not\in\Omega',\\
                            \lim_{\eps\to 0}X_{\eps}(u,t,\omega),& \omega\in\Omega',
                          \end{cases},\quad u\in(a,b),\ \ t\in(0,T],
$$
that is well-defined. By the construction of $\widetilde{X}$, $\widetilde{X}(u,t)$, $t\in(0,T]$, is continuous for all $u\in(a,b)$ and $\widetilde{X}(t,\omega)=X(t,\omega)$ (in $L_2$), $t\in(0,T]$, $\omega\in\Omega'$. Furthermore, by Proposition~\ref{prp_L2_Uparrow_crit}, \begin{equation}\label{f_monot}
\widetilde{X}(u,t,\omega)\leq \widetilde{X}(v,t,\omega),\quad u<v,\ t\in(0,T],\ \omega\in\Omega'.                                                                                                                                                                                   \end{equation}

Next, we want to extend $\widetilde{X}(u,\cdot)$ to $[0,T]$. First we will do this for all $u$ from a countable dense subset $U$ in $(a,b)$. Denote
$$
U=\{u\in(a,b):\ g\ \mbox{discontinuous at}\ u\}\cup\left((a,b)\cap\Q\right).
$$
We note that $U$ is dense in $(a,b)$ and since $g$ is a monotone function, $U$ is also countable.

Let $u\in U$ be fixed. Since $X_{\eps}(u,t)$, $t\in[0,T]$, is a continuous $(\F_t)$-square integrable martingale,
\begin{align*}
 \p\left\{\sup_{t\in[0,T]}|X_{\eps'}(u,t)-X_{\eps''}(u,t)|>r\right\}
 \leq\frac{1}{r^2}\E\left(X_{\eps'}(u,T)-X_{\eps''}(u,T)\right)^2\to 0\quad\mbox{as}\ \ \eps',\eps''\to 0,
\end{align*}
by Lemma~\ref{lem_boundedness_of_g} and the dominated convergence theorem.
Hence, $\{X_{\eps}(u,\cdot)\}_{\eps>0}$ is a Cauchy sequence in $C[0,T]$. Consequently, $X_{\eps}(u,\cdot)\to Y(u,\cdot)$ in probability, where $Y(u,\cdot)$ is a continuous process. Note that $X_{\eps}(u,0)=(g,h_{u,\eps})\to g(u)=Y(u,0)$, since $g$ is right continuous. The construction of $\widetilde{X}$ and $Y$ implies $\widetilde{X}(u,\cdot)=Y(u,\cdot)$ on $(0,T]$ a.s. and we can extend $\widetilde{X}(u,\cdot)$ to $[0,T]$ putting
$$
\widetilde{X}(u,0)=\lim_{t\to 0}\widetilde{X}(u,t)=\lim_{t\to 0}Y(u,t)=g(u)\quad \mbox{a.s.}
$$

Let
$$
\Omega''=\Omega'\cap\{\mbox{for all}\ u\in U,\ \widetilde{X}(u,t)\to g(u)\ \mbox{as}\ t\to 0\}.
$$
Since $U$ is countable, $\p\{\Omega''\}=1$.

Next, using~\eqref{f_monot}, we show that for all $\omega\in\Omega''$ and $u\in(a,b)$
\begin{equation}\label{f_lim_at_0}
\lim_{t\to 0}\widetilde{X}(u,t,\omega)=g(u).
\end{equation}
It will imply that $\widetilde{X}(u,\cdot,\omega)$ can be extended to a continuous function on $[0,T]$. Note that it is needed to check~\eqref{f_lim_at_0} only for $u\not\in U$. Here we are going to use the fact that $g$ is continuous at any $u\not\in U$. Let $r>0$, $\omega\in\Omega''$ and $v_1<u<v_2$ such that $v_1,v_2\in U$, $g(u)-g(v_1)<\frac{r}{2}$ and $g(v_2)-g(u)<\frac{r}{2}$. Next, since $\widetilde{X}(v_i,t,\omega)\in C[0,T]$, $i=1,2$, there exists $\delta>0$ such that for all $t\in(0,\delta)$
\begin{align*}
  \widetilde{X}(v_1,t,\omega)-g(v_1)>-\frac{r}{2}\quad\mbox{and}\quad
  \widetilde{X}(v_2,t,\omega)-g(v_2)<\frac{r}{2}.
\end{align*}
By the monotonicity of $\widetilde{X}(\cdot,t,\omega)$ (see~\eqref{f_monot}),
\begin{align*}
  \widetilde{X}(u,t,\omega)&-g(u)\geq\widetilde{X}(v_1,t,\omega)-g(v_1)+g(v_1)-g(u)>-\frac{r}{2}-\frac{r}{2}=-r
\end{align*}
for all $t<\delta$.
Similarly,
$$
\widetilde{X}(u,t,\omega)-g(u)<r,\quad t<\delta.
$$
It proves that $\widetilde{X}(u,t,\omega)\to g(u)$ as $t\to 0$. Thus, we can put $\widetilde{X}(u,0)=g(u)$.

Next we show that $\widetilde{X}(u,\cdot,\omega)$, $u\in(a,b)$, is right continuous in $C[0,T]$ for all $\omega\in\Omega''$. Let $u\in(a,b)$, $\omega\in\Omega''$ and $u_n\downarrow u$. It is easy to see that $\widetilde{X}(u_n,t,\omega)\downarrow \widetilde{X}(u,t,\omega)$ for all $t\in[0,T]$. Since $\widetilde{X}(u,\cdot,\omega)\in C[0,T]$, $\widetilde{X}(u_n,\cdot,\omega)\to\widetilde{X}(u,\cdot,\omega)$ in $C[0,T]$, by Dini's theorem. It implies that $\widetilde{X}(u,\cdot,\omega)$, $u\in(a,b)$, is right continuous. Similarly, we can check that $\widetilde{X}(u,\cdot,\omega)$, $u\in(a,b)$, has left limits in $C[0,T]$.

Also it should be noted that
$$
\Omega''\subseteq\{\forall u,v\in(a,b)\ \forall s\leq t\ \mbox{if } \widetilde{X}(u,s)=\widetilde{X}(v,s),\ \mbox{then } \widetilde{X}(u,t)=\widetilde{X}(v,t)\}.
$$
It proves~\eqref{f_coalescing}.

To finish the proof of the proposition, we have to show that $\widetilde{X}(u,t)$, $t\in[0,T]$, is an $(\F_t)$-square integrable martingale for all $u\in(a,b)$. First we show that $\E (\widetilde{X}(u,t))^2<\infty$ for all $u\in(a,b)$ and $t\in[0,T]$. Set
$$
\widetilde{X}^+(u,t)=\widetilde{X}(u,t)\vee 0\quad\mbox{and}\quad\widetilde{X}^-(u,t)=(-\widetilde{X}(u,t))\vee 0.
$$
By the monotonicity of $\widetilde{X}(u,t)$ in $u$, we have
\begin{align*}
  \left(\widetilde{X}^+(u,t)\right)^2&\leq\frac{1}{b-u}\int_u^b\left(\widetilde{X}^+(v,t)\right)^2dv\leq\frac{1}{b-u}\int_u^b\left(\widetilde{X}(v,t)\right)^2dv\\
  &\leq\frac{1}{b-u}\int_a^b\left(\widetilde{X}(v,t)\right)^2dv\leq\frac{1}{b-u}\|\widetilde{X}(t)\|^2_{L_2}=\frac{1}{b-u}\|X(t)\|^2_{L_2}.
\end{align*}
Similarly,
$$
\left(\widetilde{X}^-(u,t)\right)^2\leq\frac{1}{u-a}\|X(t)\|^2_{L_2}.
$$
Hence,
\begin{equation}\label{f_boundedness_of_xi}
|\widetilde{X}(u,t)|=\widetilde{X}^+(u,t)+\widetilde{X}^-(u,t)\leq\left(\frac{1}{\sqrt{b-u}}+\frac{1}{\sqrt{u-a}}\right)\|X(t)\|_{L_2}.
\end{equation}
Since $\E\|X(t)\|^2_{L_2}<\infty$, \eqref{f_boundedness_of_xi} implies $\E\left(\widetilde{X}(u,t)\right)^2<\infty$.

Next, if $s<t$, then $\E \left(\left.X_{\eps}(u,t)\right|\F_s\right)=X_{\eps}(u,s)$ and the monotone convergence theorem imply $\E \left(\left.\widetilde{X}(u,t)\right|\F_s\right)=\widetilde{X}(u,s)$. The proposition is proved.
\end{proof}

\textbf{Acknowledgements.} The author is deeply grateful to Max von Renesse for useful discussions and suggestions. The research was supported by the Alexander von Humboldt Foundation.




\providecommand{\bysame}{\leavevmode\hbox to3em{\hrulefill}\thinspace}
\providecommand{\MR}{\relax\ifhmode\unskip\space\fi MR }
\providecommand{\MRhref}[2]{%
  \href{http://www.ams.org/mathscinet-getitem?mr=#1}{#2}
}
\providecommand{\href}[2]{#2}

\end{document}